\documentclass[a4paper]{article} 
\usepackage[utf8]{inputenc}
\usepackage[left=2.3cm,right=2.3cm,top=2.9cm]{geometry}
\usepackage{authblk}
\usepackage{CJKutf8}
\usepackage{yfonts}
\usepackage{lipsum}
\title{Fisher information and trajectorial interpretation to the Itô--Langevin relative entropy dissipation}

\author{Jiaming Chen\thanks{chen.jiaming@cims.nyu.edu}}
\affil{Courant Institute of Mathematical Sciences, New York University}
\date{\today}

\usepackage{graphicx}
\usepackage{amssymb}
\usepackage{amsthm}
\usepackage{tikz}
\usepackage{mathrsfs}
\usepackage{verbatim}

\usepackage{amsmath, amsthm}

\usepackage{centernot}
\usepackage{mathtools}
\usepackage{ stmaryrd }
\usepackage{hyperref}
\hypersetup{urlcolor=blue
pdftitle={Trajectorial dynamics of relative entropy dissipation},
    pdfpagemode=FullScreen}
\numberwithin{equation}{section}
\usepackage{bbm}

\usepackage{braket}
\usepackage{amsmath,amsthm,amssymb}
\usepackage{physics}
\usepackage{mathtools}
\usepackage{bm}
\usepackage{esvect}
\usepackage[english]{babel}
\usepackage{extarrows}

\usepackage{setspace}
\usepackage{amsmath}

\numberwithin{equation}{section}
\usepackage{bbm}

\usepackage{braket}
\usepackage{amsmath,amsthm,amssymb}
\usepackage{physics}
\usepackage{mathtools}
\usepackage{bm}
\usepackage{esvect}
\usepackage[english]{babel}
\usepackage{todonotes}

\newcommand\normx[1]{\lVert#1\rVert}
\newcommand\normy[1]{\big\lVert#1\big\rVert}

\theoremstyle{definition}
\newtheorem{theorem}{Theorem}[section]

\theoremstyle{definition}
\newtheorem{corollary}[theorem]{Corollary}

\theoremstyle{definition}
\newtheorem{lemma}[theorem]{Lemma}

\theoremstyle{definition}

\theoremstyle{definition}

\theoremstyle{definition}

\theoremstyle{definition}

\usepackage[english]{babel}

\newenvironment{sequation}{\small\begin{equation}}{\end{equation}}

\begin{document}

\maketitle

\begin{abstract}

The dissipation phenomena of relative entropy from an Itô--Langevin dynamical system is a classic topic from stochastic analysis. Relying on the time-reversal of diffusions, a novel trajectorial approach investigates the pathwise behavior of relevant entropy processes, reveals more information from the delicate random structure, and eventually retrieves the known classical results. In essence, this approach provides novel insights and rederives the known results of the Itô--Langevin dynamics, as will be presented in this expository article. Another part is to view the stochastic time-evolution through the lens of the Wasserstein space, under which we observe the geometric feature of steepest descent of the entropy decay as well as its exponential rate of velocity.
    
\end{abstract}

\renewcommand{\thefootnote}{\fnsymbol{footnote}} 
\renewcommand{\thefootnote}{\arabic{footnote}} 

{
\tableofcontents}

{
\section{Introduction}\label{sec: intro}
When Schrödinger \cite{Schrödinger0} tried to explain why intelligent systems tend to have far more replication errors than general Statistical Thermodynamics, the concept of entropy was progressively developed in response to the observation that even for the most isolated physical systems, the level of their internal disorder significantly increases as time flows \cite{Abbott}.\par
This time-monotonous trend has been described as the dissipation of entropy, where the phrase was first adopted by Clausius \cite{Clausius}. In this \textbf{expository} article, we present a novel perspective to the entropy dissipation for a large class of models characterized by the Itô-Langevin stochastic differential equations.
\subsection{Reviewing the literature}
Very heuristically, the notion of entropy and its dissipation had been widely discussed by Boltzmann \cite{Boltzmann1,Boltzmann2,Boltzmann3}, Gibbs \cite{Gibbs1,Gibbs2}, and Shannon \cite{Shannon1,Shannon2}, until unanimously accepted to be the metric which quantifies the level of disorder of a physical system. Several approaches from different disciplines have been developed to characterize the monotonicity of its time-evolution. For a recent overview, see Cover/Thomas \cite{Cover/Thomas}, Krzakala/Zdeborová \cite{Krzakala/Zdeborová} for Information Theory, see Kardar \cite{Kardar1,Kardar2} for Statistical Physics, and see Sudakov \cite{Sudakov} for Combinatorics. The notion of entropy has also been utilized in Mathematical Finance, see Choulli \cite{Choulli}, Laeven/Stadje \cite{Laeven/Stadje}, and Schweizer \cite{Schweizer}. Colloquially, the monotonous time-evolution of entropy indicates some irreversible change which is intrinsic to the physical system of interest, and on average tends to increase its disorder.\par
Our discussion focuses on the concept of \textit{relative entropy} lying at the interdisciplinary realm between Stochastic Calculus and Interacting Particle Systems. The notion of relative entropy quantifies the complexity of the evolution of a family of time-parametrized probability measures, see \cite{Gaveau/Granger/Moreau/Schulman}. To summarize  precisely, we follow a probabilistic approach to formalize the Langevin diffusion as entropic gradient flux in an appropriately defined Wasserstein space. A recent breakthrough by Karatzas/Schachermayer/Tschiderer
\cite{Karatzas/Schachermayer/Tschiderer}, see also \cite{Karatzas/Tschiderer, Tschiderer/Yeung}, on the trajectorial interpretation to the essentially well-known \textit{de Bruijin identity}, see Brossier/Zozor \cite{Brossier/Zozor}, is our main theme and will be presented in the sequel.\par
In this expository article, the family of time-parametrized probability measures on $\mathbb{R}^d$ will be denoted by $(P^\beta_t)_{t\geq0}$, where the $\beta$-script indicates that this evolution is placed under the presence of some smooth and time-homogeneous perturbations. The model of interest is the trajectorial dynamics of the relative entropy of $(P^\beta_t)_{t\geq0}$ against a $\sigma$-finite reference measure $Q$ on the Borel sets of $\mathbb{R}^d$. This trajectorial approach brings us a novel interpretation to the interplay between the relative entropy and other quantities, such as Fisher information, which reveals more internal structure from the stochastic system. Over recent years, apart from the trajectorial formulation of the relative entropy dissipation, similar trajectorial approaches have been successfully applied to the optimal stopping theorems by Davis/Karatzas \cite{Davis/Karatzas}, to the Doob martingale inequalities by Acciaio/Beiglböck/Penkner/Schachermayer/Temme \cite{Acciaio/Beiglböck/Penkner/Schachermayer/Temme}, and to the Burkholder-Davis-Gundy inequality by Beiglböck/Siorpaes \cite{Beiglböck/Siorpaes}. See also Gentil/Léonard/Ripani \cite{Gentil/Léonard/Ripani} for an application in the Schrödinger problem.
\subsection{Motivation and rough descriptions}
This expository article is an interpretation to the trajectorial Otto calculus, developed by Karatzas/Scha-chermayer/Tschiderer in \cite{Karatzas/Schachermayer/Tschiderer}, when applied to the scenario of relative entropy dissipation. In essence, the trajectorial formulation provides a novel approach to the well-known phenomena of entropy dissipation \cite{Stam}, through which we are able to witness more internal information from the stochastic dynamics of interest.\par
Define $\mathcal{C}\coloneqq\mathcal{C}(\mathbb{R}_+;\mathbb{R}^d)$, consisting of $\mathbb{R}^d$-valued continuous functions on $[0,\infty)$, to be the path space where we will place the stochastic dynamics. The time-evolution of the coordinate process $(X_t(\omega))_{t\geq0}=(\omega(t))_{t\geq0}$ for all $\omega\in\mathcal{C}$ is characterized by its distribution $\mathbb{P}^\beta$ on $\mathcal{C}$, i.e.~a probability measure on the Borel sets of $\mathcal{C}$. At each time point $t\geq0$, the marginal law of $X_t$ is denoted by $P^\beta_t$, a Borel probability measure on $\mathbb{R}^d$. Collectively, we have a time-parametrized family $(P^\beta_t)_{t\geq0}$ of Borel probability measures on $\mathbb{R}^d$ at hand. And we choose a fixed $\sigma$-finite Borel reference measure $Q$ on $\mathbb{R}^d$, such that each $P^\beta_t$ is absolutely continuous with respect to $Q$, for quantifying the relative entropy. An approach to compare $P^\beta_t$ to $Q$ is computing their Radon–Nikodým derivative. It is well-known that taking $\mathbb{P}^\beta$-expectation on the logarithmic derivative $\log dP^\beta_t/dQ$ yields the classical quantity of relative entropy between $P^\beta_t$ and $Q$.\par
The first insight of the trajectorial formulation is that we work with the process $(\log dP^\beta_t/dQ)_{t\geq0}$ and investigate its trajectorial properties. This approach deals with the pathwise behavior of the relevant processes. Achieving their pathwise limiting identities and subsequently taking $\mathbb{P}^\beta$-expectation  will give us the well-known results on the dissipation phenomena on the dynamics of relative entropy. In other words, this approach will reveal more information than the classical approach from the delicate pathwise structure of the stochastic dynamical system.\par
Another insight is letting the time-parametrized family $(P^\beta_t)_{t\geq0}$ undergo time-reversal. This less transparent approach is adopted because it is comparatively simpler than the original forward-time approach, especially in the computation of the semimartingale decomposition of the relative entropy process which is defined in Section \ref{sec: semimartingale}. To phrase this backward-time approach, we fix a compact time interval $[0,T]$ and consider the same family of Borel probability measures $(P^\beta_{T-t})_{0\leq t\leq T}$, indexed backward in time. Our main object of interest, i.e.~the trajectorial formulation, originates from observing the difference between the terms $\log dP^\beta_{T-t}/dQ$ and $\mathbb{E}^{\mathbb{P}^\beta}[\log dP^\beta_{t_0}/dQ|\sigma(P^\beta_{T-\theta},0\leq\theta\leq t)]$ with $0\leq t\leq T-t_0\leq T$. Dividing this difference term by $T-t_0-t$ and letting $t\nearrow T-t_0$, we obtain formally the trajectorial time-derivative of the relative entropy process, under the conditional knowledge of $\sigma(P^\beta_{T-\theta},0\leq\theta\leq t)$. And as stated in the previous paragraph, taking $\mathbb{P}^\beta$-expectation and with some additional regularity argument, we obtain the dissipation identity of the relative entropy.\par
\subsection{Structure of this article}
From the above characterizations of the trajectorial formulation, we could retrieve the known classical results on the relative entropy dissipation by first taking $\mathbb{P}^\beta$-expectation and further collapsing the $\beta$-perturbation. Indeed, the classical \textit{de Bruijn identity} \cite[Equation 2.12]{Stam} follows immediately from the above procedures. The realization of the rough descriptions relies on the specification of the law $\mathbb{P}^\beta$ on the path space $\mathcal{C}$.\par
In Section \ref{sec: stochastic dynamics} of this expository article, we shall require $\mathbb{P}^\beta$ to be the law of the coordinate process $(X_t)_{t\geq0}$ so that it satisfies an Itô-Langevin stochastic differential equation (\ref{perturbed Itô-Langevin dynamics}), which describes a broad class of particle system dynamics, see \cite{Dougherty, Friedrich/Peinke}.  In Section \ref{sec: time-reversal}, some necessary terminologies and prerequisite theories on the time-reversal principle of diffusion processes are presented, before subsequently discussing the trajectorial formalism which heavily relies on the backward-time techniques.\par
It is essential to specify the semimartingale decomposition of the Radon–Nikodým derivative process $(dP^\beta_{T-t}/dQ)_{0\leq t\leq T}$, viewed under time-reversal, as well as the semimartingale decomposition of its logarithm $(\log dP^\beta_{T-t}/dQ)_{0\leq t\leq T}$. The computation and some related results on its martingale property will be displayed in the end of Section \ref{sec: semimartingale}. And their applications to the relative entropy dissipation are presented in Section \ref{sec: trajectorial dissipation}, where the known classical results are shown to be a derivation of the trajectorial approach.\par
In Section \ref{sec: wasserstein}, we characterize the time-evolution of the family $(P^\beta_t)_{t\geq0}$ of Borel probability measures on $\mathbb{R}^d$ through the lens of the suitably defined quadratic Wasserstein space, where its internal connections to the relative entropy dissipation via the Fisher information, as well as the steepest descent property in the unperturbed scenario, are revealed. This article is concluded with an argument of the exponential decay rate of the relative entropy quantity in the absence of perturbation, which can also be derived in parallel from the Bakry-Émery theory.


\section{The stochastic Itô--Langevin dynamics}\label{sec: stochastic dynamics}
Ever since the seminal contribution \cite{Langevin} to the Brownian motion theory, the Itô--Langevin stochastic differential equations have played an eminent role in the non-equilibrium Statistical Mechanics \cite{Oliveira/Tomé, Ornstein/Uhlenbeck} and in the study of particle systems \cite{Gardiner, Schuss}. The fundamental idea of Itô-Langevin dynamics is to describe the diffusion particle in terms of the combination of deterministic forces and stochastic fluctuations.\par
In this expository article, we will also place the entropy dynamics under the constraint of an Itô--Langevin stochastic differential equation. To express the spirit of our trajectorial formulation and taking into account the conciseness of this exposition, we will focus on the simplest setting of a particle undergoing diffusion in a potential field. Notice that throughout our exposition, $\abs{\cdot}:\mathbb{R}\to\mathbb{R}_+$ denotes the absolute value of a real number, and $\norm{\cdot}:\mathbb{R}^d\to\mathbb{R}_+$ denotes the Euclidean $L^2$-norm of a vector in $\mathbb{R}^d$. \par
\subsection{Itô--Langevin dynamics} Denote by $\psi(\cdot):\mathbb{R}^d\to\mathbb{R}_+$ the potential function, which is assumed to be of class $\mathcal{C}^\infty(\mathbb{R}^d;\mathbb{R}_+)$ and satisfies the linear growth condition $\normx{\nabla\psi(x)}\leq K(1+\norm{x})$ for all $x\in\mathbb{R}^d$ with some absolute constant $K>0$. This potential function determines the distribution $\mathbb{P}^\beta$ on $\mathcal{C}=\mathcal{C}(\mathbb{R}_+;\mathbb{R}^d)$ and henceforth also the time-evolution of the family of marginal distributions $(P^\beta_t)_{t\geq0}$ of the coordinate process $(X_t)_{t\geq0}$. Furthermore, starting from a fixed time point $t_0\geq0$, a smooth perturbation field $\beta(\cdot):\mathbb{R}^d\to\mathbb{R}^d$ started to influence the evolution. To capture the aforementioned setting, the Itô-Langevin stochastic differential equation,
\begin{equation}
    \label{perturbed Itô-Langevin dynamics}
    dX_t=-\big(\nabla\psi(X_t)+\beta(X_t)I_{\{t>t_0\}}\big)\,dt+dW^\beta_t,\quad\text{for all}\quad t\geq0\quad\text{with}\quad X_{0}\sim P_{0},
\end{equation}
constrains the time-evolution of the coordinate process $(X_t)_{t\geq0}$ in $\mathcal{C}$, and hence also the family $(P^\beta_t)_{t\geq0}$ of Borel probability measures on $\mathbb{R}^d$. Here, $(W^\beta_t)_{t\geq0}$ is a 
$d$-dimensional Brownian motion started from zero, independent of $X_0$. In (\ref{perturbed Itô-Langevin dynamics}), the initial distribution $P_0$ of $X_0$ is put to be absolutely continuous with respect to the Lebesgue measure on $\mathbb{R}^d$. Throughout this expository article, we assume that the perturbation $\beta(\cdot)$ is of compact support. Indeed, this regularity requirement simplifies our argument, and we have
\begin{lemma}
    \label{lem: strong solution, perturbed Itô-Langevin}
    The Itô-Langevin diffusion (\ref{perturbed Itô-Langevin dynamics}) with initial distribution $P_0$ admits a pathwise unique strong solution $(X_t)_{t\geq0}$, whose distribution on $\mathcal{C}$ is denoted by $\mathbb{P}^\beta$. If we assume that the distribution $P_0$ of $X_0$ admits a finite second moment, i.e.~$\mathbb{E}^{\mathbb{P}^\beta}[\norm{X_0}^2]<\infty$ and that the potential function $\psi(\cdot)$ satisfies the following drift condition,
    \begin{equation}\label{drift condition}
        x\cdot\nabla\psi(x)\geq-C\norm{x}^2,\quad\forall~x\in\mathbb{R}^d\quad\text{with}\quad\norm{x}\geq R,\quad\text{for some}\quad C,R>0,
    \end{equation}
    then each $X_t$, with $t\geq0$, admits a finite second moment, i.e.~$\mathbb{E}^{\mathbb{P}^\beta}[\norm{X_t}^2]<\infty$.
\end{lemma}
\begin{proof}
    Since the potential $\psi(\cdot)$ is of class $\mathcal{C}^\infty(\mathbb{R}^d;\mathbb{R}_+)$, its gradient $\nabla\psi(\cdot)$ is then of class $\mathcal{C}^\infty(\mathbb{R}^d;\mathbb{R}^d)$. Hence, both $(\nabla\psi+\beta)(\cdot)$ and $\nabla\psi(\cdot)$ are locally Lipschitz continuous on compact sets of $\mathbb{R}^d$, because $\beta(\cdot)$ is smooth with compact support. Therefore, together with the linear growth condition on $\nabla\psi(\cdot)$, the existence of a strong solution $(X_t)_{t\geq0}$ in $\mathcal{C}$ and its pathwise uniqueness is a consequence of Le Gall \cite[Theorem 8.3]{Le Gall}. We are left to verify that each $X_t$ admits a finite second moment. Indeed, in (\ref{drift condition}) we can choose $R>0$ sufficiently large so that $\text{supp}(\beta)\subseteq\{x\in\mathbb{R}^d:\,\norm{x}\leq R\}$. Then, denote
    \begin{equation*}
        C_R\coloneqq\sup\big\{d-2x\cdot (\nabla\psi+\beta I_{\{t>t_0\}})(x):\,\norm{x}\leq R,\,t\geq0\big\}\quad\text{and}\quad\tau_k\coloneqq\inf\big\{t\geq0:\,\norm{X_t}>k\big\},\;\;\forall~k>R.
    \end{equation*}
    Notice that the drift condition (\ref{drift condition}) guarantees $C_R<\infty$. Then the Itô formula gives
    \begin{equation*}
        d\norm{X_t}^2=\big(d-2X_t\cdot (\nabla\psi+\beta I_{\{t>t_0\}})(X_t)\big)\,dt+2X_t\,dW^\beta_t,\quad\text{for all}\quad t\geq0.
    \end{equation*}
    Here, both $X_t$ and $W^\beta_t$ are vectors in $\mathbb{R}^d$, so integrating $X$ against $W^\beta$ simply refers to $\sum_{i=1}^d\int X_t^{(i)}\,dW_t^{\beta,(i)}$. Taking $\mathbb{P}^\beta$-expectation under the localization sequence $(\tau_k)_{k>R}$, we observe that
    \begin{equation*}\begin{aligned}
        \mathbb{E}^{\mathbb{P}^\beta}\big[\norm{X_{\tau_k\wedge t}}^2\big]&=\mathbb{E}^{\mathbb{P}^\beta}\big[\norm{X_{0}}^2\big]+\mathbb{E}^{\mathbb{P}^\beta}\bigg[\int_0^{\tau_k\wedge t}\big(d-2X_\theta\cdot (\nabla\psi+\beta I_{\{\theta>t_0\}})(X_\theta)\big)I_{\{\norm{X_\theta}\leq R\}}\,d\theta\bigg]\\
        &\quad+\mathbb{E}^{\mathbb{P}^\beta}\bigg[\int_0^{\tau_k\wedge t}\big(d-2X_\theta\cdot (\nabla\psi+\beta I_{\{\theta>t_0\}})(X_\theta)\big)I_{\{\norm{X_\theta}> R\}}\,d\theta\bigg]\\
        &\leq\mathbb{E}^{\mathbb{P}^\beta}\big[\norm{X_{0}}^2\big]+C_R\mathbb{E}^{\mathbb{P}^\beta}\big[\tau_k\wedge t\big]+\mathbb{E}^{\mathbb{P}^\beta}\bigg[\int_0^{\tau_k\wedge t}d+2C\norm{X_\theta}^2\,d\theta\bigg].
    \end{aligned}\end{equation*}
    The last inequality above follows from (\ref{drift condition}) and the definition of $C_R$. Hence,
    \[
        \mathbb{E}^{\mathbb{P}^\beta}\big[\norm{X_{\tau_k\wedge t}}^2\big]\leq\mathbb{E}^{\mathbb{P}^\beta}\big[\norm{X_{0}}^2\big]+(C_R+d)t+2C\int_0^t \mathbb{E}^{\mathbb{P}^\beta}\big[\norm{X_{\tau_k\wedge\theta}}^2\big]\,d\theta\quad\text{for all}\quad t\geq0.
    \]
    Applying the Gronwall inequality \cite[Section 2]{Gronwall}, we obtain
    \begin{equation}\label{Gronwall inequality}
        \mathbb{E}^{\mathbb{P}^\beta}\big[\norm{X_{\tau_k\wedge t}}^2\big]\leq\mathbb{E}^{\mathbb{P}^\beta}\big[\norm{X_{0}}^2\big]+(C_R+d)t+2C\int_0^t e^{2C(t-\theta)}\big(\mathbb{E}^{\mathbb{P}^\beta}\big[\norm{X_{0}}^2\big]+(C_R+d)\theta\big)\,d\theta.
    \end{equation}
    The assumption $\mathbb{E}^{\mathbb{P}^\beta}[\norm{X_{0}}^2]<\infty$ implies that the RHS of (\ref{Gronwall inequality}) is finite for all $t\geq0$. Applying the monotone convergence theorem \cite[Theorem 1.26]{Rudin} and letting $k\nearrow\infty$,
    \begin{equation*}
        \mathbb{E}^{\mathbb{P}^\beta}[\norm{X_{t}}^2]\leq\mathbb{E}^{\mathbb{P}^\beta}\big[\norm{X_{0}}^2\big]+(C_R+d)t+2C\int_0^t e^{2C(t-\theta)}\big(\mathbb{E}^{\mathbb{P}^\beta}\big[\norm{X_{0}}^2\big]+(C_R+d)\theta\big)\,d\theta<\infty,
    \end{equation*}
    for all $t\geq0$, which verifies the claim.
\end{proof}
Lemma \ref{lem: strong solution, perturbed Itô-Langevin} tells us that the second moment condition propagates in time. Indeed, from now on we assume that $\mathbb{E}^{\mathbb{P}^\beta}[\norm{X_0}^2]<\infty$, which automatically implies
\begin{equation}\label{temp}
    \mathbb{E}^{\mathbb{P}^\beta}\big[\norm{X_t}^2\big]=\int_{\mathbb{R}^d}\norm{x}^2\,dP^\beta_t\in\mathbb{R},\quad\text{for all}\quad t\geq0.
\end{equation}
In Section \ref{sec: wasserstein},we will endow the Wasserstein space structure to the set of all Borel probability measures on $\mathbb{R}^d$ with finite second moments. There, (\ref{temp}) shows that $P^\beta_t$, $t\geq0$, belongs to this Wasserstein space, whose metric structure provides more insights, for instance Theorem \ref{w-space displacement}, on the evolution of $(P^\beta_t)_{t\geq0}$.
\subsection{Density and reference measure}
The absolute continuity of each $P^\beta_t$ with respect to the Lebesgue measure on $\mathbb{R}^d$ is guaranteed in \cite[Section 2]{Jordan/Kinderlehrer/Otto}. For all $t\geq0$, we write $p^\beta_t(\cdot):\mathbb{R}^d\to\mathbb{R}_+$ as the density of $P^\beta_t$ against the Lebesgue measure on $\mathbb{R}^d$. Notice that $(t,x)\mapsto p^\beta_t(x)$ satisfies a partial differential equation on $\mathbb{R}_+\times\mathbb{R}^d$, called the \textit{Fokker-Planck equation}. This more analytic perspective will be discussed when we compute the semimartingale decomposition of $(dP^\beta_{T-t}/dQ)_{0\leq t\leq T}$ and $(\log dP^\beta_{T-t}/dQ)_{0\leq t\leq T}$ in Section \ref{sec: semimartingale}, where the time-reversal is performed in a compact interval $[0,T]$. For now, we only remark that the Fokker-Planck equation plays an important role in classical dissipative systems \cite{Gardiner, Risken} and has internal connections to the Itô-Langevin dynamics \cite{Schuss}.\par
The family of probability measures $(P^\beta_t)_{t\geq0}$ then has two interpretations: Either its density $(p^\beta_t(\cdot))_{t\geq0}$ can be seen as the solution to the Fokker-Planck equation, to be written out in Section \ref{sec: semimartingale}, or each $P^\beta_t$ can be seen as the marginal distribution to the solution process $(X_t)_{t\geq0}$ of the Itô-Langevin dynamics (\ref{perturbed Itô-Langevin dynamics}) at time  $t\geq0$. Apart from that, we introduce a $\sigma$-finite measure $Q$ on the Borel sets of $\mathbb{R}^d$ with density
\begin{equation*}
    q(\cdot)\coloneqq\exp\big(-2\psi(\cdot)\big):\mathbb{R}^d\to\mathbb{R}_+
\end{equation*}
against the Lebesgue measure on $\mathbb{R}^d$. This $\sigma$-finite Borel measure $Q$ is specified to be the reference measure when we compute the relative entropy of $(P^\beta_t)_{t\geq0}$ later in this article.\par
When the potential function $\psi(\cdot)$ grows rapidly enough so that $\exp(-\psi(\cdot))\in L^2(\mathbb{R}^d)$, the normalized density $q(\cdot)$ solves a variant Fokker-Plank equation, if we take $p^\beta_0(\cdot)=q(\cdot)$ modulo normalization. Looking back to (\ref{perturbed Itô-Langevin dynamics}), an equivalent probabilistic perspective reveals that $X_t\sim Q$ modulo normalization at each time $t\geq0$, see \cite{Gardiner, Risken}. Moreover, when $\exp(-\psi(\cdot))\in L^2(\mathbb{R}^d)$, it is verified \cite[Section 4]{Jordan/Kinderlehrer} that the normalized probability density $q(\cdot)$ satisfies a variational principle: It minimizes the \textit{free energy functional},
\begin{equation*}
    \mathscr{F}(\rho)\coloneqq\int_{\mathbb{R}^d}\psi(x)\rho(x)\,dx+\frac{1}{2}\int_{\mathbb{R}^d}\rho(x)\log\rho(x)\,dx,
\end{equation*}
over all probability densities $\rho(\cdot):\mathbb{R}^d\to\mathbb{R}_+$ on $\mathbb{R}^d$.\par
Additionally, we require the perturbation field to be of gradient type, i.e.~$\beta(\cdot)=\nabla B(\cdot):\mathbb{R}^d\to\mathbb{R}^d$ in our exposition. The function $B(\cdot):\mathbb{R}^d\to\mathbb{R}$, of class $\mathcal{C}^\infty(\mathbb{R}^d,\mathbb{R})$ and compactly supported, is called the perturbation potential. When the perturbation is switched off, the family of probability measures characterizing the coordinate process $(X_t)_{t\geq0}$ from (\ref{perturbed Itô-Langevin dynamics}) is denoted by $(P^0_t)_{t\geq0}$, where the zero-script simply indicates that this is the case of vanishing perturbation.
\subsection{Quantities from statistical physics} So far we have been working with the family of Borel probability measures $(P^\beta_t)_{t\geq0}$ and the $\sigma$-finite Borel reference measure $Q$ on $\mathbb{R}^d$. But how do we extract information from their time-evolution? One approach is to translate our language and work through the lens of stochastic processes.\par
It is now clear that each $P^\beta_t$ is absolutely continuous with respect to the $\sigma$-finite reference measure $Q$, for any $t\geq0$. Indeed, we write the \textit{likelihood ratio process}, or the Radon–Nikodým derivative, as 
\begin{equation}
    \label{likelihood ratio process}
    \ell^{\mathbb{P}^\beta}_t(X_t)=\frac{dP^\beta_t}{dQ},\quad\text{where}\quad\ell^{\mathbb{P}^\beta}_t(x)\coloneqq p^\beta_t(x)e^{2\psi(x)}\quad\text{for all}\quad(t,x)\in\mathbb{R}_+\times\mathbb{R}^d.
\end{equation}
And we call its logarithmic process as the \textit{relative entropy process},
\begin{equation}
    \label{relative entropy process}
    \mathcal{R}^{\mathbb{P}^\beta}_t(X_t)\coloneqq\log\ell^{\mathbb{P}^\beta}_t(X_t)=\log\frac{dP^\beta_t}{dQ}\;\;\text{for all}\;\;t\geq0.
\end{equation}
This seemly redundant definition will actually simplify computation in the analysis of semimartingale decomposition of $(\ell^{\mathbb{P}^\beta}_{T-t}(X_{T-t}))_{0\leq t\leq T}$ and $(\mathcal{R}^{\mathbb{P}^\beta}_{T-t}(X_{T-t}))_{0\leq t\leq T}$ in Section \ref{sec: semimartingale}. For notational conciseness, from now on we will abbreviate $\ell^{\mathbb{P}^\beta}$ and $\mathcal{R}^{\mathbb{P}^\beta}$ as $\ell^\beta$ and $\mathcal{R}^\beta$, respectively.\par
Having defined the basic setting of relevant stochastic processes, we now introduce some quantities of interest. These quantities have their origins from Statistical Physics \cite{Kardar2} and will, essentially, serve as the metrological index of our Itô-Langevin stochastic system. Regarding the family $(P^\beta_t)_{t\geq0}$ of Borel probability measures and the $\sigma$-finite reference measure $Q$, we define the \textit{relative entropy},
\begin{equation}
    \label{relative entropy}
    \mathbb{H}\big[P^\beta_t|Q\big]\coloneqq\int_{\mathbb{R}^d}\log\frac{dP^\beta_t}{dQ}\,dP^\beta_t=\int_{\mathbb{R}^d}p^\beta_t(x)\log\frac{p^\beta_t(x)}{q(x)}\,dx,\quad\text{for all}\quad t\geq0,
\end{equation}
as well as the \textit{Fisher information},
\begin{equation}
    \label{Fisher information}
    \mathbb{I}\big[P^\beta_t|Q\big]\coloneqq\int_{\mathbb{R}^d}\normy{\nabla\log\frac{dP^\beta_t}{dQ}}^2\,dP^\beta_t=\int_{\mathbb{R}^d}\normy{\nabla\big(\log p_t^\beta(x)+2\psi(x)\big)}^2p^\beta_t(x)\,dx,\quad\text{for all}\quad t\geq0.
\end{equation}\par
To avoid a meticulous discussion on the general case, some regularity is assumed for the time-evolution of the relative entropy. And therefore, our argument is simplified to better present the trajectorial formulation in Section \ref{sec: trajectorial dissipation}. We add the assumption that our choice of the initial distribution $P_0$ on $\mathbb{R}^d$ ensures $\mathbb{H}[P_0|Q]<\infty$. Incorporating $\ell^\beta$ and $\mathcal{R}^\beta$ into the above definitions, the relative entropy (\ref{relative entropy}) and Fisher information (\ref{Fisher information}) can then be written as
\begin{equation*}
    \mathbb{H}\big[P^\beta_t|Q\big]=\mathbb{E}^{\mathbb{P}^\beta}\big[\mathcal{R}^\beta_t(X_t)\big]\quad\text{and}\quad\mathbb{I}\big[P^\beta_t|Q\big]=\mathbb{E}^{\mathbb{P}^\beta}\big[\normy{\nabla\mathcal{R}^\beta_t(X_t)}^2\big],\quad\text{for all}\quad t\geq0.
\end{equation*}\par
One remarkable consequence of defining the relative entropy of probability measure $P^\beta_t$ with respect to the $\sigma$-finite Borel measure $Q$ is that the mapping $t\mapsto\mathbb{H}[P^0_t|Q]$ admits a strong version of monotonicity on $\mathbb{R}_+$, in the absence of perturbation.
\begin{lemma}\label{relative entropy H is Q-submartingale}
    For finite time horizon $T\geq0$ and let $\tau_1,\tau_2$ with $\tau_1\leq\tau_2$ be two stopping times taking value in $[0,T]$ with respect to the filtration generated by the coordinate process $(X_t)_{t\geq0}$, then,
    \begin{equation*}
        \mathbb{H}\big[P^0_{T-\tau_1}|Q\big]\leq\mathbb{H}\big[P^0_{T-\tau_2}|Q\big]\quad\text{and}\quad\mathbb{H}\big[P^0_{T-t_1}|Q\big]\leq\mathbb{H}\big[P^0_{T-t_2}|Q\big],\quad\text{for all}\quad0\leq t_1\leq t_2\leq T.
    \end{equation*}
\end{lemma}
Lemma \ref{relative entropy H is Q-submartingale} is actually a surface corollary of an internal property of the relative entropy process: $(\mathcal{R}^\beta_t)_{t\geq0}$ running at backward-time satisfies the condition to be a $Q$-submartingale, where the reference measure $Q$ is only required to be $\sigma$-finite on $\mathbb{R}^d$, not necessarily a probability measure. For the precise definition of a $Q$-submartingale and for the proof of Lemma \ref{relative entropy H is Q-submartingale}, readers are referred to \cite[Corollary 1.25]{Tschiderer}.\par
Remember that $(P^0_t)_{t\geq0}$ is the family of marginal distributions induced by (\ref{perturbed Itô-Langevin dynamics}), without perturbation. In fact, $P^\beta_t$ coincides with $P^0_t$ when $0\leq t\leq t_0$, before the perturbation $\beta(\cdot)$ is initiated. Similarly, we denote by $\mathbb{H}[P^0_t|Q]$, $\mathbb{I}[P^0_t|Q]$, and $\ell^0_t(X_t)$, $\mathcal{R}^0_t$ for all $t\geq0$, in their respective zero-perturbation case. The above-defined relative entropy and Fisher information provide decisive metric to a quantitative version of the motivation and rough descriptions mentioned in Section \ref{sec: intro}. Such trajectorial formulation will be thoroughly investigated in Section \ref{sec: trajectorial dissipation}.
\subsection{Preview of classical results} The trajectorial approach which is presented in Section \ref{sec: trajectorial dissipation} of this expository article reveals more internal information from the Itô-Langevin stochastic dynamics, but its formulation is rather abstract and difficult to comprehend at first reading. Therefore, it will be a courtesy to present the known classical results on the relative entropy dissipation to the readers as a flashing lamp, before we follow a long journey comprising of the time-reversal principles in Section \ref{sec: time-reversal} and the semimartingale decomposition of $\ell^\beta$ and $\mathcal{R}^\beta$ in Section \ref{sec: semimartingale}, which eventually leads to the trajectorial approach in Section \ref{sec: trajectorial dissipation}.\par
The classical result on relative entropy dissipation is phrased as the time-derivative of the relative entropy (\ref{relative entropy}), in the absence of perturbation,
\begin{equation*}
    \lim\limits_{t\searrow t_0}\frac{1}{t-t_0}\bigg(\mathbb{H}[P^0_t|Q]-\mathbb{H}[P^0_{t_0}|Q]\bigg)=-\frac{1}{2}\mathbb{I}[P^0_{t_0}|Q]\;\;\text{for all}\;\;t_0\geq0,
\end{equation*}
which renders us the well-known \textit{de Bruijin identity} \cite[Equation 2.12]{Stam}. An observation of the above limiting identity tells us that the time-derivative of relative entropy is therefore expressed as the Fisher information modulo a multiplicative constant. Another quantity of interest is to estimate the limiting behavior of the metric from the Borel probability measure $P^\beta_t$ to $P^\beta_{t_0}$ on $\mathbb{R}^d$, as $t\searrow t_0$.\par
To give a more precise description to this metric, we define $\mathscr{P}_2(\mathbb{R}^d)$ to be the \textit{quadratic Wasserstein space}, whose elements consist of all probability measures on $\mathbb{R}^d$ admitting a finite second moment. And the space $\mathscr{P}_2(\mathbb{R}^d)$ is equipped with the suitably defined \textit{quadratic Wasserstein metric} $W_2(\mu,\nu)$ for all $\mu,\nu\in\mathscr{P}_2(\mathbb{R}^d)$. For now we just view $W_2$ as a well-defined metric on $\mathscr{P}_2(\mathbb{R}^d)$. Its exact definition as well as the detailed discussion of the quadratic Wasserstein space will be deferred to Section \ref{sec: wasserstein}. We are interested in the limiting behavior of $W_2(P^\beta_t,P^\beta_{t_0})$ as $t\searrow t_0$. In Section \ref{sec: wasserstein}, it will be shown that
\begin{equation*}
        \lim\limits_{t\searrow t_0}\frac{1}{t-t_0}W_2(P^0_t,P^0_{t_0})=\frac{1}{2}\sqrt{\mathbb{I}[P^0_{t_0}|Q]}.
\end{equation*}\par
Clearly, the limiting time-derivative of the relative entropy dissipation and that of the quadratic Wasserstein metric are strongly correlated in the sense that
\begin{equation*}
    \lim\limits_{t\searrow t_0}\frac{\,\mathbb{H}[P^0_t|Q]-\mathbb{H}[P^0_{t_0}|Q]\,}{W_2(P^0_t,P^0_{t_0})}=-\sqrt{\mathbb{I}[P^0_{t_0}|Q]},
\end{equation*}
which reveals also the fact that Fisher information (\ref{Fisher information}) serves as a bridge-gate between the relative entropy and the quadratic Wasserstein metric. The idea to consider the relative entropy dissipation in the context of quadratic Wasserstein space was first discussed by Jordan/Kinderlehrer/Otto \cite{Jordan/Kinderlehrer/Otto} and Otto \cite{Otto}.\par
This expository article takes into consideration an external deterministic perturbation, i.e.~there is no extra randomness governing the perturbation field, to the unperturbed Itô-Langevin dynamics. This is a natural extension to the known results, but its major importance is the revelation of the so called \textit{steepest descent property} of the relative entropy dissipation. The steepest descent property can only be precisely described after we have presented our analysis of the time-displacement of $(P^\beta_t)_{t\geq0}$ viewed from the quadratic Wasserstein space perspective in Section \ref{sec: wasserstein}. And this steepest descent property answers the question why the unperturbed dynamics is remarkably different from the the same stochastic systems placed under the smooth perturbation field $\beta(\cdot)$.\par
To achieve the trajectorial formulation of the relative entropy dissipation, it is more convenient to look at things backward in time. The following Section \ref{sec: time-reversal} presents no new results, but it contains all the necessary background theory of stochastic processes under time-reversal. 
    

\section{Time-reversal of diffusion processes}\label{sec: time-reversal}
As announced in the end of Section \ref{sec: stochastic dynamics}, this preparatory section contains ramifications on the theory of time-reversal principles of diffusions. We choose to present this general topic before discussing the semimartingale decomposition of relevant processes $\ell^\beta$ and $\mathcal{R}^\beta$ in Section \ref{sec: semimartingale} as well as the trajectorial formulation in Section \ref{sec: trajectorial dissipation}, because many time-reversal techniques are adopted to formalize the main results, which are conveniently written in a backward-time fashion. For a pedagogical reasoning, a courtesy on various filtrations, Wiener processes, and Itô integration under a backward-time approach becomes quite necessary.
\subsection{Historical comments}
 The principle of time-reversal in stochastic analysis has a distinguished history in many disciplines of sciences. This type of question has been of interest to physicists, most notably Guerra/Marra \cite{Guerra/Marra}, Nelson \cite{Nelson}, and Witten \cite{Witten, Witten2} as well as to control theorists Lindquist/Picci \cite{Lindquist/Picci}, Goussev/Jalabert/Pastawski/Wis-niacki \cite{Goussev/Jalabert/Pastawski/Wisniacki}. The philosophy of time-reversal principles has also shed light to economists, see Zumbach \cite{Zumbach}. Previous to our work, the connection between time-reversal dynamics and the Itô-Langevin stochastic differential equations has also been discussed in \cite{Anderson, Chen/Margarint,Chen/Margarint2, Pardoux}.\par
 It is well-known that Markov process remains a Markov process under time-reversal \cite[Section 1]{Haussmann/Pardoux}. However, the strong Markovian property is not necessarily preserved under time-reversal \cite[Chapter V.7]{Rogers/Williams}, and neither is the semimartingale property \cite[Section 1]{Walsh}. So it is of interest to see whether the diffusion property, i.e.~strong Markovian semimartingale property, is preserved under time-reversal.\par
Instead of analyzing the Itô-Langevin dynamics (\ref{perturbed Itô-Langevin dynamics}), we start with a general $\mathbb{R}^d$-valued diffusion process, i.e.~strong Markovian continuous semimartingale, $(S_t)_{t\geq0}$ driven by a stochastic differential equation, see for instance (\ref{forward diffusion equation}), with smooth drift and dispersion coefficients. Our main goal in this section is to assertion that its time-reversed process,
\begin{equation}
    \label{time-reversed process}
    \widehat{S}_t\coloneqq S_{T-t}\;\;\text{for all}\;\;0\leq t\leq T,
\end{equation}
is a diffusion, adapted to a backward filtration which will be specified later, provided sufficient regularity on its constraint stochastic differential equation, for instance (\ref{forward diffusion equation}). Such question goes back to Boltzmann \cite{Boltzmann2, Boltzmann3, Boltzmann4}, Schrödinger \cite{Schrödinger1, Schrödinger2}, and Kolmogorov \cite{Kolmogorov}. Time-reversal of stochastic processes was dealt with systematically by Nelson \cite{Nelson}, Carlen \cite{Carlen} in the context of dynamical theory for diffusions. It was developed in the context of filtering, interpolation and extrapolation by Haussmann/Pardoux \cite{Haussmann/Pardoux} and Pardoux \cite{Pardoux}. In a non-Markovian context, the time-reversal of diffusions was developed by Föllmer \cite{Föllmer1, Föllmer2}. See also Margarint \cite{Margarint} and Napolitano/Sakurai \cite{Napolitano/Sakurai} for the time-reversal principles applied to Mathematical Physics.\par
In this expository article, we focus on the time-reversal principles relevant to the Itô-Langevin stochastic differential equation (\ref{perturbed Itô-Langevin dynamics}) and demonstrate that the time-reversal of its solution process maintains the diffusion property, provided sufficient regularity conditions on its drift and dispersion terms, under a suitable filtered probability space. Henceforth, in Sections \ref{sec: semimartingale} and \ref{sec: trajectorial dissipation} where we formalize the trajectorial interpretation of the relative entropy dissipation, it becomes safe to wielding the time-reversal techniques. Moreover, it is convenient to restrict our discussion to a compact time horizon $T>0$ without loss of generality.
\subsection{Backward filtrations}
Under time-reversal, the backward processes are no longer adapted to the original forward-time filtrations. Consequently, it is necessary to construct some new filtrations, from the known information, which expand backward in time. For a reference on the theory of filtrations, readers are referred to Protter \cite{Protter}. We place a filtered probability space $(\Omega,\mathcal{F},\mathbb{F},\mathbb{P})$ with the forward-time filtration $\mathbb{F}\coloneqq(\mathcal{F}_t)_{0\leq t\leq T}$, where
\begin{equation*}
    \mathcal{F}_t\coloneqq\sigma(\xi,W_\theta:\,0\leq\theta\leq t)\quad\text{for all}\quad0\leq t\leq T,
\end{equation*}
modulo $\mathbb{P}$-augmentation. Here $\xi$ is an $\mathcal{F}_0$-measurable and $(W_t)_{0\leq t\leq T}$ is an $\mathbb{F}$-Brownian motion starting from zero, independent of $\xi$. Next, consider the stochastic differential equations
\begin{equation}
    \label{forward diffusion equation}
    S^{(i)}_t=\xi^{(i)}+\int_0^ta_i(\theta,S_\theta)\,d\theta+\sum\limits_{\nu=1}^m\int_0^tb_{i\nu}(\theta,S_\theta)\,dW^{(\nu)}_\theta\quad\text{for all}\quad0\leq t\leq T,
\end{equation}
with $i=1,2,\ldots,d$. We assume that (\ref{forward diffusion equation}) admits a pathwise unique strong solution, which conforms to our Itô-Langevin setting of the coordinate process $(X_t)_{0\leq t\leq T}$, where its strong existence and the pathwise uniqueness property is verified in Lemma \ref{lem: strong solution, perturbed Itô-Langevin}. Then, $S=(S^{(i)},\ldots,S^{(d)})^T$ is $\mathbb{F}$-adapted and
\begin{equation*}
    \mathcal{F}_t=\sigma(S_\theta,W_\theta:\,0\leq\theta\leq t)=\sigma(S_0,W_t-W_\theta:\,0\leq\theta\leq t)\quad\text{for all}\quad0\leq t\leq T,
\end{equation*}
modulo $\mathbb{P}$-augmentation. It follows that $(S_t)_{0\leq t\leq T}$ has the $(\mathcal{F}_t)_{0\leq t\leq T}$-strong Markovian property, see \cite[Section 5.2]{Karatzas/Shreve}.\par
We further assume that the drifts $a_i(t,x)$ and dispersions $b_{i\nu}(t,x)$ are of class $\mathcal{C}^\infty(\mathbb{R}_+\times\mathbb{R}^d;\mathbb{R})$, for all $1\leq i\leq d$ and $1\leq\nu\leq m$. Hence their regularity contains enough smoothness. And the covariance matrix $\sigma(t,x)$ of (\ref{forward diffusion equation}), given by
\begin{equation*}
    \sigma_{ij}(t,x)\coloneqq\sum\limits_{\nu=1}^mb_{i\nu}(t,x)b_{j\nu}(t,x),\quad\text{for all}\quad1\leq i,j\leq d,
\end{equation*}
is of class $\mathcal{C}^\infty(\mathbb{R}_+\times\mathbb{R}^d;\mathbb{R}^{d\times d})$. The density function $\rho_t(\cdot):\mathbb{R}^d\to\mathbb{R}_+$ of the marginal law of $S_t$ against the Lebesgue measure on $\mathbb{R}^d$ solves the forward Kolmogorov equation \cite[Chapter 3]{Erban/Chapman}, \cite[Section 5.7]{Karatzas/Shreve},
\begin{equation*}
    \frac{\partial \rho_t}{\partial t}(x)=\frac{1}{2}\sum\limits_{1\leq i,j\leq d}\frac{\partial^2}{\partial x_i\partial x_j}\big(\sigma_{ij}(t,x)\rho_t(x)\big)-\sum\limits_{1\leq i\leq d}\frac{\partial}{\partial x_i}\big(a_i(t,x)\rho_t(x)\big),\quad\text{for all}\quad(t,x)\in[0,T]\times\mathbb{R}^d.
\end{equation*}
Define the following filtration $\widehat{\mathbb{F}}\coloneqq(\widehat{\mathcal{F}}_{T-t})_{0\leq t\leq T}$ running backward in time, by
\begin{equation}
    \label{backward filtration F}
    \widehat{\mathcal{F}}_{T-t}\coloneqq\sigma(S_{T-\theta},W_{T-\theta}-W_{T-t}:\,0\leq\theta\leq t)\quad\text{for all}\quad0\leq t\leq T.
\end{equation}
For each $0\leq t\leq T$, This $\sigma$-algebra $\widehat{\mathcal{F}}_{T-t}$ can be equivalently expressed as
\begin{equation*}\begin{aligned}
    \widehat{\mathcal{F}}_{T-t}&=\sigma(S_{T-t},W_{T-\theta}-W_{T-t}:\,0\leq\theta\leq t)= \sigma(S_{T-t},W_{T}-W_{T-\theta}:\,0\leq\theta\leq t)\\
    &=\sigma(S_{T},W_{T-t}-W_{T-\theta}:\,0\leq\theta\leq t)=\sigma(S_T)\vee\mathcal{H}_{T-t},
\end{aligned}\end{equation*}
where $\mathcal{H}_{T-t}\coloneqq\sigma(W_{T-t}-W_{T-\theta}:\,0\leq\theta\leq t)$ is independent of the random vector $S_{T-t}$, for all $0\leq t\leq T$, see \cite[Section 4]{Liang/Lyons/Qian}. Then the backward-time process $(\widehat{S}_t)_{0\leq t\leq T}$ and $\widetilde{W}_t\coloneqq W_{T-t}-W_T$, $0\leq t\leq T$ are both adapted to $\widehat{\mathbb{F}}$ defined in (\ref{backward filtration F}). Indeed, the $\sigma$-algebra $\widehat{\mathcal{F}}_{T-t}$ can be further expressed as
\begin{equation*}
    \widehat{\mathcal{F}}_{T-t}=\sigma(\widehat{S}_\theta,\widetilde{W}_\theta-\widetilde{W}_t:\,0\leq\theta\leq t)=\sigma(\widehat{S}_0)\vee\mathcal{H}_{T-t},\quad
    \text{where}\;\;\mathcal{H}_{T-t}=\sigma(\widetilde{W}_\theta-\widetilde{W}_t:\,0\leq\theta\leq t).
\end{equation*}\par
The notion of a backward filtration is essential, once we put the time-evolution under a reversed direction. To have a meaningful discussion on the relevant backward-time stochastic processes, it is necessary to specify the backward-time filtrations to which these processes are filtered. In consequence, the above argument provides necessary supplements to guarantee this point, when we write down the semimartingale decomposition of $\ell^\beta$ as well as $\mathcal{R}^\beta$ backward in time in Section \ref{sec: semimartingale}, and when we formulate the trajectorial approach of the relative entropy dissipation written in a backward-time fashion in Section \ref{sec: trajectorial dissipation}.
\subsection{Wiener process and Itô integration}
Running filtrations under time-reversal induces a new question. How do we identify an adapted backward-time Wiener process? Indeed, it is possible that the time-reversal of a forward-time Brownian motion loses its martingale property under a backward filtration. Nonetheless, if we subtract a proper backward-time finite variation process, the Lévy theorem \cite[Theorem 5.12]{Le Gall} yields the adapted Brownian motion under time-reversal.
\begin{lemma}\label{backward Brownian motion}
    The backward-time process $(\widetilde{W}_t)_{0\leq t\leq T}$ is a Brownian motion of its own filtration $(\mathcal{H}_{T-t})_{0\leq t\leq T}$, but only a semimartingale to the strictly larger filtration $\widehat{\mathbb{F}}$. On the other hand, if we define the backward-time process $(B_t)_{0\leq t\leq T}=((B_t^{(1)}\ldots,B_t^{(m)})^T)_{0\leq t\leq T}$ by
    \begin{equation}\label{definition of B}
        B^{(\nu)}_t\coloneqq\widetilde{W}^{(\nu)}_t-\int_0^t\rho_{T-\theta}^{-1}\sum\limits_{1\leq i\leq d}\frac{\partial}{\partial x_i}\big(\rho_{T-\theta}(\cdot)b_{i\nu}(T-\theta,\cdot)\big)(\widehat{S}_\theta)\,d\theta,\quad\text{for all}\quad0\leq t\leq T,
    \end{equation}
    with $\nu=1,2,\ldots,m$. Then $(B_t)_{0\leq t\leq T}$ is an $\mathbb{R}^m$-valued $\widehat{\mathbb{F}}$-adapted Brownian motion independent of $\widehat{\mathcal{F}}_T$, and therefore also independent of $S_T$.
\end{lemma}
\begin{proof}
    First, we need to show that each component $B^{(\nu)}$, $\nu=1,\ldots,m$ of the backward-time process $B$ is a $\widehat{\mathbb{F}}$-adapted martingale. In other words, for all bounded $\widehat{\mathcal{F}}_t$-measurable $\mathcal{K}$, we have to show that
    \begin{equation}\label{how to show B is martingale}
        \mathbb{E}^{\mathbb{P}}\big[\big(B^{(\nu)}_{T-\theta}-B^{(\nu)}_{T-t}\big)\mathcal{K}\big]=0,\quad\text{for all}\quad0\leq\theta\leq t\leq T.
    \end{equation}
    Since $\mathbb{E}^{\mathbb{P}}[\mathcal{K}|\mathcal{F}_t]=\mathbb{E}^{\mathbb{P}}[\mathcal{K}|S_t]$ $\mathbb{P}$-a.s. there exists a Borel measurable $K_t:\mathbb{R}^m\to\mathbb{R}$ such that $K_t(S_t)=\mathbb{E}^{\mathbb{P}}[\mathcal{K}|\mathcal{F}_t]$. We further define $K_\theta(x)\coloneqq\mathbb{E}^{\mathbb{P}}[K_t(S_t)|S_\theta=x]$ for all $(\theta,x)\in[0,t]\times\mathbb{R}^m$. Invoking the Markovian property of $(S_t)_{0\leq t\leq T}$ and following the ideas from Meyer \cite{Meyer}, we deduce that the process,
    \begin{equation*}
        K_\theta(S_\theta)=\mathbb{E}^{\mathbb{P}}\big[K_t(S_t)|S_\theta\big]=\mathbb{E}^{\mathbb{P}}\big[\mathcal{K}|\mathcal{F}_\theta\big],\quad\text{for all}\quad0\leq\theta\leq t,
    \end{equation*}
    is an $\mathbb{F}$-martingale, and therefore,
    \begin{equation*}
        K_t(S_t)-K_\theta(S_\theta)=\sum\limits_{1\leq i\leq d}\sum\limits_{1\leq\nu\leq m}\int_\theta^t\frac{\partial K_\tau}{\partial x_i}(S_\tau)b_{i\nu}(\tau,S_\tau)\,dW^{(\nu)}_\tau.
    \end{equation*}
    Since $\mathbb{E}^{\mathbb{P}}[(W^{(\nu)}_t-W^{(\nu)}_\theta)K_t(S_t)]=\mathbb{E}^{\mathbb{P}}[(W^{(\nu)}_t-W^{(\nu)}_\theta)(K_t(S_t)-K_\theta(S_\theta))]$,
    \begin{equation*}
        \mathbb{E}^{\mathbb{P}}\big[\big(W^{(\nu)}_t-W^{(\nu)}_\theta\big)K_t(S_t)\big]=\mathbb{E}^{\mathbb{P}}\big[\sum\limits_{i=1}^d\int_\theta^t\frac{\partial K_\tau}{\partial x_i}(S_\tau)b_{i\nu}(\tau,S_\tau)\,d\tau\big]=\sum\limits_{i=1}^d\int_\theta^t\int_{\mathbb{R}^d}\big(b_{i\nu}(\tau,\cdot)\frac{\partial K_\tau}{\partial x_i}\big)(x)\rho_\tau(x)\,dx\,d\tau.
    \end{equation*}
    Integrating by parts, for each $\nu=1,\ldots,m$, this yields,
    \begin{equation}\begin{aligned}\label{consequence, why B is martingale}
        &\;\;\;\;-\mathbb{E}^{\mathbb{P}}\big[\big(W^{(\nu)}_t-W^{(\nu)}_\theta\big)K_t(S_t)\big]=\sum\limits_{i=1}^d\int_\theta^t\int_{\mathbb{R}^d}K_\tau(x)\frac{\partial}{\partial x_i}\big(\rho_\tau(\cdot)b_{i\nu}(\tau,\cdot)\big)(x)\,dx\,d\tau\\
        &=\int_\theta^t\mathbb{E}^{\mathbb{P}}\big[K_\tau(S_\tau)\rho^{-1}_\tau\sum\limits_{i=1}^d\frac{\partial}{\partial x_i}\big(\rho_\tau(\cdot)b_{i\nu}(\tau,\cdot)\big)(S_\tau)\big]\,d\tau=\mathbb{E}^{\mathbb{P}}\big[K_t(S_t)\int_\theta^t\rho^{-1}_\tau\sum\limits_{i=1}^d\frac{\partial}{\partial x_i}\big(\rho_\tau(\cdot)b_{i\nu}(\tau,\cdot)\big)(S_\tau)\,d\tau\big].
    \end{aligned}\end{equation}
    Combining (\ref{definition of B}) and (\ref{consequence, why B is martingale}), we get
    \begin{equation*}
        \mathbb{E}^{\mathbb{P}}\bigg[\mathbb{E}^{\mathbb{P}}\big[\mathcal{K}|\mathcal{F}_t\big]\bigg(W^{(\nu)}_t-W^{(\nu)}_\theta+\int_\theta^t\rho_\tau^{-1}\sum\limits_{i=1}^d\frac{\partial}{\partial x_i}\big(\rho_\tau(\cdot)b_{i\nu}(\tau,\cdot)\big)(S_\tau)\,d\tau\bigg)\bigg]=0,\quad\text{for all}\quad0\leq\theta\leq t\leq T,
    \end{equation*}
    which is equivalent to (\ref{how to show B is martingale}), where the conditional expectation can be removed because both $W^{(\nu)}_t-W^{(\nu)}_\theta$ and $(S_\tau)_{\thefootnote\leq\tau\leq t}$ are $\mathcal{F}_t$-measurable. Hence, $(B^{\nu}_t)_{0\leq t\leq T}$ is a $\widehat{\mathbb{F}}$-martingale for each $\nu=1,\ldots,m$. In view of the continuity of the sample paths and the property,
    \begin{equation*}
        \big\langle B^{(\mu)}, B^{(\nu)}\big\rangle_t=\big\langle \widetilde{W}^{(\mu)}, \widetilde{W}^{(\nu)}\big\rangle_t=t\delta_{\mu\nu}\quad\text{for all}\quad1\leq\mu,\nu\leq m\quad\text{and}\quad0\leq t\leq T,
    \end{equation*}
    we can infer that each $(B^{\nu}_t)_{0\leq t\leq T}$ is a $\widehat{\mathbb{F}}$-Brownian motion such that $B^{(\mu)}$ and $B^{(\nu)}$ are mutually independent for all $\mu\neq\nu$, by appealing to Lévy theorem \cite[Theorem 5.12]{Le Gall}. Henceforth, $(B_t)_{0\leq t\leq T}$ is a $\mathbb{R}^m$-valued  $\widehat{\mathbb{F}}$-Brownian motion.
\end{proof}
Our main goal of this section is to verify that the time-reversal $(\widehat{S}_t)_{0\leq t\leq T}$ is a diffusion process, under some suitable backward filtrations. In fact, we furthermore specify its semimartingale decomposition in Lemma \ref{backward S is diffusion}. To achieve this goal, we need to introduce a notion of backward stochastic integration which uses finite sums of backward increments to approximate the stochastic integrals. Such scheme is essential to the proof of Lemma \ref{backward S is diffusion}, and is called the \textit{backward Itô integration}.\par
Consider two continuous semimartingales $X_t=X_0+M_t+K_t$ and $Y_t=Y_0+N_t+L_t$, where $(M_t)_{0\leq t\leq T}$ and $(N_t)_{0\leq t\leq T}$ are continuous local martingales, $(K_t)_{0\leq t\leq T}$ and $(L_t)_{0\leq t\leq T}$ are continuous finite variation processes. By analogy with its forward-time counterpart, the \textit{backward} Itô integral \cite{Le Gall} is defined by,
\begin{equation}
    \label{backward Itô integral}
    \int_0^tY_\theta\bullet dX_\theta\coloneqq\int_0^tY_\theta\,dM_\theta+\int_0^tY_\theta\,dK_\theta+\langle M,N\rangle_t,\quad\text{for all}\quad0\leq t\leq T.
\end{equation}
If $\Pi=\{t_0=0,t_1,\ldots,t_m=T\}$ is a partition of the time interval $[0,T]$, denote by $\norm{\Pi}\coloneqq\max\{t_{j}-t_{j-1}:\,1\leq j\leq m\}$. And we have the following convergence in probability \cite{Russo/Vallois},
\begin{equation*}
    \sum\limits_{0\leq j\leq m-1}Y_{t_{j+1}}(X_{t_{j+1}}-X_{t_j})\xlongrightarrow{\,\mathbb{P}\,}\int_0^TY_t\bullet dX_t,\quad\text{as}\quad\norm{\Pi}\to0.
\end{equation*}
And for all $f\in\mathcal{C}^2(\mathbb{R}^d;\mathbb{R})$, the backward Itô integral admits the change of variable formula \cite{Prato/Menaldi/Tubaro}, \cite{Russo/Vallois},
\begin{equation*}
    f(X_t)=f(X_0)+\sum\limits_{1\leq i\leq d}\int_0^t\frac{\partial f}{\partial x_i}(X_\theta)\bullet dX^{(i)}_\theta-\frac{1}{2}\sum\limits_{1\leq i,j\leq d}\int_0^t\frac{\partial^2f}{\partial x_i\partial x_j}(X_\theta)\,d\langle M^{(i)},M^{(j)}\rangle_\theta,
\end{equation*}
for all $0\leq t\leq T$. Following the definition of the backward Itô integral, in Lemma \ref{backward S is diffusion}, we show that the time-reversal $(\widehat{S}_t)_{0\leq t\leq T}$ is indeed a $\widehat{\mathbb{F}}$-diffusion.
\subsection{Diffusions under time-reversal}
The main goal of this section says that a forward-time diffusion, with sufficient regularity on its drift and dispersion coefficients, remains a diffusion process under time-reversal with respect to a suitable backward filtration. It is important because we perform a time-reversal technique to the backward-time semimartingale decomposition of $\ell^\beta$ and $\mathcal{R}^\beta$ in Lemmas \ref{lem: sem. dec. l} and \ref{lem: sem. dec. R}, as well as to the trajectorial formulation of relative entropy dissipation in Theorems \ref{thm: Trajectorial rate of relative entropy dissipation, time-displacement, perturbed} and \ref{thm: Trajectorial rate of relative entropy dissipation, derivative}.
\begin{lemma}\label{backward S is diffusion}
    Given a $\mathbb{R}^d$-valued diffusion process $(S_t)_{0\leq t\leq T}$ adapted to $(\mathcal{F}_{t})_{0\leq t\leq T}$, define its time-reversal $(\widehat
    {S}_t)_{0\leq t\leq T}$ as in (\ref{time-reversed process}) and define the backward filtration $\widehat{\mathbb{F}}$ as in (\ref{backward filtration F}). Then, $(\widehat
    {S}_t)_{0\leq t\leq T}$ is an $\widehat{\mathbb{F}}$-adapted diffusion, i.e.~a strong Markovian semimartingale, with the decomposition,
    \begin{equation}\label{backward-time S}
        \widehat{S}^{(i)}_t=\widehat{S}^{(i)}_0+\int_0^t\widehat{a}_i(T-\theta,\widehat{S}_\theta)\,d\theta+\sum\limits_{\nu=1}^m\int_0^tb_{i\nu}(T-\theta,\widehat{S}_\theta)\,dB^{(\nu)}_\theta\quad\text{for all}\quad0\leq t\leq T,
    \end{equation}
    where for each $i=1,\ldots,d$,
    \begin{equation*}
        \widehat{a}_i(t,x)\coloneqq\sum\limits_{1\leq j\leq d}\frac{\partial\sigma_{ij}}{\partial x_j}(t,x)+\sum\limits_{1\leq j\leq d}\sigma_{ij}(t,x)\frac{\partial}{\partial x_j}\log \rho_t(x)-a_i(t,x),\quad\text{for all}\quad(t,x)\in[0,T]\times\mathbb{R}^d.
    \end{equation*}
\end{lemma}
\begin{proof}
    From (\ref{forward diffusion equation}) and by the Itô formula, the process
    \begin{equation*}
        b_{i\nu}(t,S_t)-b_{i\nu}(0,\xi)-\sum\limits_{1\leq j\leq d}\sum\limits_{1\leq\kappa\leq m}\int_0^t\big(b_{j\kappa}(\theta,\cdot)\frac{\partial b_{i\nu}}{\partial x_j}(\theta,\cdot)\big)(S_\theta)\,dW^{(\kappa)}_\theta
    \end{equation*}
    is of finite variation. Hence,
    \begin{sequation}\label{quadratic variation}
        \big\langle b_{i\nu}(\cdot,S),W^{(\nu)}\big\rangle_t=\sum\limits_{1\leq j\leq d}\int_0^t\big(b_{j\nu}(\theta,\cdot)\frac{\partial b_{i\nu}}{\partial x_j}(\theta,\cdot)\big)(S_\theta)\,d\theta\quad\text{for all}\quad0\leq t\leq T.
    \end{sequation}
    On the other hand, we can express the forward-time diffusion $(S_t)_{0\leq t\leq T}$ in terms of backward Itô integral,
    \begin{equation*}
        S^{(i)}_t-\xi^{(i)}-\int_0^ta_i(\theta,S_\theta)\,d\theta=\sum\limits_{\nu=1}^m\int_0^tb_{i\nu}(\theta,S_\theta)\bullet dW^{(\nu)}_\theta-\big\langle b_{i\nu}(\cdot,S),W^{(\nu)}\big\rangle_t.
    \end{equation*}
    Combining with (\ref{quadratic variation}), we observe that
    \begin{equation*}
        S^{(i)}_t=\xi^{(i)}-\int_0^t\bigg(\sum\limits_{j=1}^d\sum\limits_{\nu=1}^m b_{j\nu}(\theta,\cdot)\frac{\partial b_{i\nu}}{\partial x_j}(\theta,\cdot)-a_i(\theta,\cdot)\bigg)(S_\theta)\,d\theta+\sum\limits_{\nu=1}^m\int_0^tb_{i\nu}(\theta,S_\theta)\bullet dW^{(\nu)}_\theta.
    \end{equation*}
    Evaluating also at the terminal time point $T$, this gives
    \begin{equation*}
        S^{(i)}_t=S^{(i)}_T+\int_t^T\bigg(\sum\limits_{j=1}^d\sum\limits_{\nu=1}^m b_{j\nu}(\theta,\cdot)\frac{\partial b_{i\nu}}{\partial x_j}(\theta,\cdot)-a_i(\theta,\cdot)\bigg)(S_\theta)\,d\theta-\sum\limits_{\nu=1}^m\int_t^Tb_{i\nu}(\theta,S_\theta)\bullet dW^{(\nu)}_\theta,
    \end{equation*}
    as well as
    \begin{equation*}
        \widehat{S}_t=\widehat{S}_0+\int_0^t\bigg(\sum\limits_{j=1}^d\sum\limits_{\nu=1}^m b_{j\nu}(T-\theta,\cdot)\frac{\partial b_{i\nu}}{\partial x_j}(T-\theta,\cdot)-a_i(T-\theta,\cdot)\bigg)(\widehat{S}_\theta)\,d\theta+\sum\limits_{\nu=1}^m\int_0^tb_{i\nu}(T-\theta,\widehat{S}_\theta)\, d\widetilde{W}^{(\nu)}_\theta.
    \end{equation*}
    via time-reversal. In light of (\ref{definition of B}), we could write the time-reversed process $(\widehat{S}_t)_{0\leq t\leq T}$ into
    \begin{equation*}\begin{aligned}
        \widehat{S}_t&=\widehat{S}_0+\sum\limits_{\nu=1}^m\int_0^tb_{i\nu}(T-\theta,\widehat{S}_\theta)\,dB^{(\nu)}_\theta+\int_0^t\bigg(\sum\limits_{j=1}^d\sum\limits_{\nu=1}^m b_{j\nu}(T-\theta,\cdot)\frac{\partial b_{i\nu}}{\partial x_j}(T-\theta,\cdot)\bigg)(\widehat{S}_\theta)\,d\theta\\
        &\quad+\int_0^t\bigg(\sum\limits_{\nu=1}^m\rho^{-1}_{T-\theta}(\cdot)b_{i\nu}(T-\theta,\cdot)\sum\limits_{j=1}^d\frac{\partial}{\partial x_j}\big(\rho_{T-\theta}(\cdot)b_{j\nu}(T-\theta,\cdot)\big)-a_i(T-\theta,\cdot)\bigg)(\widehat{S}_\theta)\,d\theta,
    \end{aligned}\end{equation*}
    which provides a semimartingale decomposition for the $\widehat{\mathbb{F}}$-adapted process $(\widehat
    {S}_t)_{0\leq t\leq T}$.\par
    The strong Markovian property follows from the existence of a strong solution (in our assumption) of (\ref{forward diffusion equation}) which is pathwise unique, see Le Gall \cite[Corollary 8.8]{Le Gall}. Hence, the time-reversal $(\widehat
    {S}_t)_{0\leq t\leq T}$ is a $\widehat{\mathbb{F}}$-diffusion process.
    Calculating the drift coefficients gives us (\ref{backward-time S}).
\end{proof}
Under sufficient regularity conditions, the Lemma \ref{backward S is diffusion} shows that the time-reversal of a diffusion process remains a diffusion, adapted to a suitable backward filtration. This lemma paves the way to the investigation of many relevant processes backward in time, which also gives us their semimartingale decomposition.


\section{Semimartingale decomposition of $\ell^\beta$ and $\mathcal{R}^\beta$}\label{sec: semimartingale}
The aim of this section is to provide a semimartingale decomposition to the likelihood ratio process $\ell^\beta$ (\ref{likelihood ratio process}) and its logarithm, the relative entropy process $\mathcal{R}^\beta$ (\ref{relative entropy process}), both running backward in time.  Later in Lemma \ref{lem: M is L2 martingale}, it is verifies that the local martingale part from the semimartingale decomposition of $(\mathcal{R}^\beta_{T-t}(X_{T-t}))_{0\leq t\leq T}$ is in fact a square integrable martingale, adapted to a suitable backward filtration. This martingale property allows us to take $\mathbb{P}^\beta$-expectation to the backward-time $\mathcal{R}^\beta$ without invoking the localization sequence of stopping times, and henceforth retrieves the relative entropy quantity.\par
After presenting the general principles of time-reversal of diffusions in Section \ref{sec: time-reversal}, let us turn our attention to the Itô-Langevin dynamics (\ref{perturbed Itô-Langevin dynamics}). The semimartingale decomposition of $\mathcal{R}^\beta$, and hence also of $\ell^\beta$, requires some knowledge of the differential structure of the likelihood ratio $\ell^\beta_t(x)=p^\beta_t(x)e^{2\psi(x)}$, $(t,x)\in\mathbb{R}_+\times\mathbb{R}^d$. First, we write down the partial differential equation which is satisfied by the density function $p^\beta_t(x)$, $(t,x)\in\mathbb{R}_+\times\mathbb{R}^d$. This type of partial differential equation is called the Fokker-Planck equation, which is internally connected to the Itô-Langevin dynamics, as stated in Section \ref{sec: stochastic dynamics}.
\subsection{Fokker--Planck equation}
In Section \ref{sec: stochastic dynamics}, we denote by $\mathbb{P}^\beta$ the distribution on $\mathcal{C}=\mathcal{C}(\mathbb{R}_+;\mathbb{R}^d)$ of the strong solution process $(X_t)_{t\geq0}$ of the Itô-Langevin stochastic differential equation (\ref{perturbed Itô-Langevin dynamics}). At each $t\geq0$, we use $P^\beta_t$ to denote the law of the marginal $X_t$. Each $P^\beta_t$ is absolutely continuous with respect to the Lebesgue measure on $\mathbb{R}^d$ and thus induces a probability density $p^\beta_t(\cdot):\mathbb{R}^d\to\mathbb{R}_+$. The Itô--Langevin dynamics is internally connected to the Fokker---Planck equation in that $p^\beta_t(\cdot)$ satisfies the partial differential equation \cite{Jordan/Kinderlehrer/Otto},
\begin{equation}
    \label{perturbed Fokker-Plank}
    \frac{\partial p^\beta_t}{\partial t}(x)=\sum\limits_{j=1}^d\frac{\partial}{\partial x_j}\bigg(\big(\frac{\partial\psi}{\partial x_j}(x)+\beta(x)I_{\{t>t_0\}}\big)p^\beta_t(x)\bigg)+\frac{1}{2}\sum\limits_{j=1}^d\frac{\partial^2p^\beta_t}{\partial x_j^2}(x)\quad\text{with}\quad p^\beta_{0}(\cdot)=p_{0}(\cdot),
\end{equation}
for all $(t,x)\in\mathbb{R}_+\times\mathbb{R}^d$. Here $p^\beta_0(\cdot)=p_0(\cdot)$ is the density function of the initial distribution $P_0\sim X_0$ to (\ref{perturbed Itô-Langevin dynamics}) against the Lebesgue measure on $\mathbb{R}^d$. The existence and uniqueness of a solution to (\ref{perturbed Fokker-Plank}) is guaranteed, see \cite[Section 4]{Ji/Qi/Shen/Yi}. The solution $p^\beta_t(\cdot)$ conserves its $L^1(\mathbb{R}^d)$ norm \cite{Karatzas/Schachermayer/Tschiderer}, which means that
\begin{equation*}
    \int_{\mathbb{R}^d}p^\beta_t(x)\,dx\equiv1,\quad\text{for all}\quad(t,x)\in\mathbb{R}_+\times\mathbb{R}^d.
\end{equation*}
And this conservation principle confirms that $(p^\beta_t(\cdot))_{t\geq0}$ is indeed a family of probability densities on $\mathbb{R}^d$.\par
An observation of (\ref{perturbed Fokker-Plank}) tells us that the time-evolution of $(P^\beta_t)_{t\geq0}$, or equivalently of $(p^\beta_t(\cdot))_{t\geq0}$, is governed by the real-valued function $\psi(\cdot)$ when $0\leq t\leq t_0$, and additionally by the $\mathbb{R}^d$-valued perturbation $\beta(\cdot)$ when $t>t_0$. If the perturbation is switched off, i.e.~$\beta(\cdot)$ vanishes, the family of probability measures generated from (\ref{perturbed Fokker-Plank}) is denoted by $(P^0_t)_{t\geq0}$ with their densities denoted by $(p^0_t(\cdot))_{t\geq0}$. Here, the zero-script simply indicates that this is the case of vanishing perturbation.\par
To write down the relative entropy (\ref{relative entropy}), we have introduced a $\sigma$-finite reference measure $Q$ on the Borel sets of $\mathbb{R}^d$. This reference measure $Q$ is defined via its density function $q(\cdot)=\exp(-2\psi(\cdot)):\mathbb{R}^d\to\mathbb{R}_+$ against the Lebesgue measure on $\mathbb{R}^d$. In contrast to the densities $(p^\beta_t(\cdot))_{t\geq0}$, the density $q(\cdot)$ is time-invariant and solves the stationary version of the Fokker-Plank equation \cite{Jordan/Kinderlehrer/Otto},
\begin{equation}
    \label{Fokker-Plank for Q}
    \sum\limits_{j=1}^d\frac{\partial}{\partial x_j}\big(\frac{\partial\psi}{\partial x_j}(\cdot)q(\cdot)\big)(x)+\frac{1}{2}\sum\limits_{j=1}^d\frac{\partial^2q}{\partial x_j^2}(x)=0,\quad\text{for all}\quad x\in\mathbb{R}^d.
\end{equation}\par
Some literature also name the equations (\ref{perturbed Fokker-Plank}) and (\ref{Fokker-Plank for Q}) as the forward-Kolmogorov equations. But they refer to the same thing. Remember that we have defined the relative entropy process $(\mathcal{R}^\beta_t(X_t))_{0\leq t\leq T}=(\log\ell^\beta_t(X_t))_{0\leq t\leq T}$ (\ref{relative entropy process}) via the function $\ell^\beta_t(x)=p^\beta_t(x)/q(x)$. Therefore, to understand the semimartingale decomposition of the processes $\ell^\beta$ and $\mathcal{R}^\beta$, either forward-time or backward-time, we need to characterize the differential structure of the densities $(p^\beta_t(\cdot))_{0\leq t\leq T}$ and $q(\cdot)$ as in (\ref{perturbed Fokker-Plank}) and (\ref{Fokker-Plank for Q}).

\subsection{Filtration and time-displacement}
 At this stage, it becomes important to specify the relevant filtrations. We denote by $(\mathcal{F}_t)_{t\geq0}$ the smallest forward continuous filtration to which the Brownian motion $(W^\beta_t)_{t\geq0}$ and the solution process $(X_t)_{t\geq0}$ of (\ref{perturbed Itô-Langevin dynamics}) is adapted. That is,
\begin{equation*}
    \mathcal{F}_t\coloneqq\sigma(X_\theta,W^\beta_\theta:\,0\leq\theta\leq t),\quad\text{for all}\quad t\geq0
\end{equation*}
modulo $\mathbb{P}^\beta$-augmentation. Likewise, given the compact time interval $[0,T]$, we denote by $(\mathcal{G}_{T-t})_{0\leq t\leq T}$ the backward continuous filtration generated by the backward processes $(W^\beta_{T-t})_{t\geq0}$ and $(X_{T-t})_{0\leq t\leq T}$. That is,
\begin{equation*}
    \mathcal{G}_{T-t}\coloneqq\sigma(X_{T-\theta},W^\beta_{T-\theta}:\,0\leq\theta\leq t),\quad\text{for all}\quad 0\leq t\leq T.
\end{equation*}
Notice that we are using similar notations for the forward-time and backward-time filtrations as in Section \ref{sec: time-reversal}. We hope this convention will leave no ambiguity because these two filtrations $(\mathcal{F}_t)_{0\leq t\leq T}$ and $(\mathcal{G}_{T-t})_{0\leq t\leq T}$ are essentially constructed in almost the same way to the filtrations in Section \ref{sec: time-reversal}, except that the filtrations here are generated by the solution process $(X_t)_{t\geq0}$, rather than a general diffusion process $(S_t)_{t\geq0}$.\par
Even though $(W^\beta_t)_{t\geq0}$ is a $(\mathcal{F}_t)_{t\geq0}$-adapted Brownian motion running forward in time, its time-reversal $(W^\beta_{T-t})_{0\leq t\leq T}$ is not necessarily a backward-time Brownian motion adapted to $(\mathcal{G}_{T-t})_{0\leq t\leq T}$. It turns out that this time-reversal process contains a non-trivial finite variation part in its semimartingale decomposition. And to construct a true $(\mathcal{G}_{T-t})_{0\leq t\leq T}$-Brownian motion backward in time, we need to subtract this finite variation process.
\begin{lemma}
    In the Itô-Langevin dynamics (\ref{perturbed Itô-Langevin dynamics}), $(W^\beta_t)_{t\geq0}$ is denoted to be the $d$-dimensional Brownian motion.
    The backward-time process
    \begin{equation}
        \overline{W}^{\mathbb{P}^\beta}_{T-t}\coloneqq W^\beta_{T-t}-W^\beta_T-\int_0^t\nabla\log p_{T-\theta}^\beta(X_{T-\theta})\,d\theta,\quad\text{for all}\quad0\leq t\leq T
    \end{equation}
    is a Brownian motion of the backward filtration $(\mathcal{G}_{T-t})_{0\leq t\leq T}$ under $\mathbb{P}^\beta$. Furthermore, the time-reversed process $(X_{T-t})_{0\leq t\leq T}$ is a $(\mathcal{G}_{T-t})_{0\leq t\leq T}$-diffusion process with its semimartingale decomposition given by
    \begin{equation}\label{semimartingale decomposition of X}
        dX_{T-t}=\nabla\log p_{T-t}^\beta(X_{T-t})\,dt+(\nabla\psi+\beta I_{\{0\leq t<T-t_0 \}})(X_{T-t})\,dt+d\overline{W}^{\mathbb{P}^\beta}_{T-t},\quad\text{for all}\quad0\leq t\leq T,
    \end{equation}
    with respect to the backward filtration $(\mathcal{G}_{T-t})_{0\leq t\leq T}$.
\end{lemma}
\begin{proof}
    To verify that $(\overline{W}^{\mathbb{P}^\beta}_{T-t})_t$ is a $(\mathcal{G}_{T-t})_{0\leq t\leq T}$-Brownian motion, we use Lemma \ref{backward Brownian motion}. And to verify that $(X_{T-t})_{0\leq t\leq T}$ is a $(\mathcal{G}_{T-t})_{0\leq t\leq T}$-diffusion process with its semimartingale decomposition (\ref{semimartingale decomposition of X}), we use Lemma \ref{backward S is diffusion}. And then the assertion is verified.
\end{proof}
Having selected the suitable backward filtration and Brownian motions, we look at the semimartingale decomposition of the likelihood ratio process $(\ell^\beta_{T-t}(X_{T-t}))_t$ and of the relative entropy process $(\mathcal{R}^\beta_{T-t}(X_{T-t}))_t$, whose pathwise behavior is the essence to the trajectorial formulation of relative entropy dissipation.
\subsection{Semimartingale decomposition}
Looking back to the forward-time coordinate process $(X_t)_{t\geq0}$ on $\mathcal{C}$ characterized by (\ref{perturbed Itô-Langevin dynamics}), we aim for the semimartingale decomposition of the processes $\ell^\beta$ and $\mathcal{R}^\beta$ running under time-reversal. This will be the first step to understand the trajectorial formulation of relative entropy dissipation in Section \ref{sec: trajectorial dissipation}.
\begin{lemma}\label{lem: sem. dec. l}
    The backward-time likelihood ratio process $(\ell^\beta_{T-t}(X_{T-t}))_{0\leq t\leq T}$ is a semimartingale adapted to $(\mathcal{G}_{T-t})_{0\leq t\leq T}$ with decomposition
    \begin{equation}\begin{aligned}\label{semimartingale decomposition of l}
        d\ell^\beta_{T-t}(X_{T-t})&=\big(2\beta\cdot\nabla\psi-\sum\limits_{1\leq i\leq d}\frac{\partial\beta^{(i)}}{\partial x_i}\big)(X_{T-t})\ell^\beta_{T-t}(X_{T-t})I_{\{0\leq t<T-t_0\}}\,dt\\
        &\quad+\normy{\nabla\ell^\beta_{T-t}(X_{T-t})}^2\,dt+\nabla\ell^\beta_{T-t}(X_{T-t})\,d\overline{W}^{\mathbb{P}^\beta}_{T-t},\quad\text{for all}\quad0\leq t\leq T.
    \end{aligned}\end{equation}
\end{lemma}
\begin{proof}
    Since $\ell^\beta_{T-t}(X_{T-t})=dP^\beta_{T-t}/dQ$, we know $\ell^\beta_{T-t}(\cdot)=p^\beta_{T-t}(\cdot)+\exp(2\psi(\cdot))$. From (\ref{perturbed Fokker-Plank}), (\ref{Fokker-Plank for Q}), we can compute that
    \begin{equation*}
        \frac{\partial\ell^\beta_{T-t}}{\partial t}(x)=-\frac{1}{2}\Delta\ell^\beta_{T-t}(x)+\nabla\ell^\beta_{T-t}\cdot\big(\nabla\psi-\beta I_{\{0\leq t<T-t_0\}}\big)(x)+\big(2\beta\cdot\nabla\psi-\sum\limits_{1\leq i\leq d}\frac{\partial\beta^{(i)}}{\partial x_i}\big)(x)\ell^\beta_{T-t}(x)I_{\{0\leq t<T-t_0\}}.
    \end{equation*}
    Applying (\ref{semimartingale decomposition of X}) and invoking the Itô formula, (\ref{semimartingale decomposition of l}) follows, and the assertion is verified.
\end{proof}
Lemma \ref{lem: sem. dec. l} gives us a $(\mathcal{G}_{T-t})_{0\leq t\leq T}$ semimartingale decomposition of the backward-time likelihood ratio process $(\ell^\beta_{T-t})_{0\leq t\leq T}$. Since taking its logarithm produces the relative entropy process $(\mathcal{R}^\beta_{T-t})_{0\leq t\leq T}$, we write the following derivation.
\begin{lemma}\label{lem: sem. dec. R}
    The backward-time relative entropy process $(\mathcal{R}^\beta_{T-t}(X_{T-t}))_{0\leq t\leq T}$ is a semimartingale adapted to $(\mathcal{G}_{T-t})_{0\leq t\leq T}$ with decomposition
    \begin{equation}\begin{aligned}\label{semimartingale decomposition of R}
        d\mathcal{R}^\beta_{T-t}(X_{T-t})&=\big(2\beta\cdot\nabla\psi-\sum\limits_{i=1}^d\frac{\partial\beta^{(i)}}{\partial x_i}\big)(X_{T-t})I_{\{0\leq t<T-t_0\}}\,dt\\
        &\quad+\frac{1}{2}\normy{\nabla\mathcal{R}^\beta_{T-t}(X_{T-t})}^2\,dt+\nabla\mathcal{R}^\beta_{T-t}(X_{T-t})\,d\overline{W}^{\mathbb{P}^\beta}_{T-t},\quad\text{for all}\quad0\leq t\leq T.
    \end{aligned}\end{equation}
\end{lemma}
\begin{proof}
    Remember that $\mathcal{R}^\beta_{T-t}(X_{T-t})=\log\ell^\beta_{T-t}(X_{T-t})$. Applying (\ref{semimartingale decomposition of l}) and invoking the Itô formula, (\ref{semimartingale decomposition of R}) follows, and the assertion is verified.
\end{proof}
The semimartingale decomposition (\ref{semimartingale decomposition of R}) splits $(\mathcal{R}^\beta_{T-t}(X_{T-t}))_{0\leq t\leq T}$ into a sum of a local martingale and a finite variation process, adapted to the filtration $(\mathcal{G}_{T-t})_{0\leq t\leq T}$. In the following, we will verify that the local martingale part is actually a martingale. This property allows us to take $\mathbb{P}^\beta$-expectation to $\mathcal{R}^\beta$ and henceforth cancel the martingale part without employing a localization sequence of stopping times. This scheme yields an expression of the relative entropy (\ref{relative entropy}) using Fisher information (\ref{Fisher information}).\par
For the clarity of this exposition, we introduce some new notations. Denote the backward-time \textit{cumulative Fisher information process} by
\begin{equation}
    \label{cumulative Fisher information process}
    \mathcal{F}^\beta_{T-t}\coloneqq\int_0^t\big(2\beta\cdot\nabla\psi-\sum\limits_{i=1}^d\frac{\partial\beta^{(i)}}{\partial x_i}\big)(X_{T-\theta})I_{\{0\leq \theta<T-t_0\}}+\frac{1}{2}\normy{\nabla\mathcal{R}^\beta_{T-\theta}(X_{T-\theta})}^2\,d\theta,
\end{equation}
for all $0\leq t\leq T$, which is of finite variation and adapted to the filtration $(\mathcal{G}_{T-t})_{0\leq t\leq T}$. Simultaneously, we denote the backward-time local martingale by,
\begin{equation}\label{perturbed square-integrable martingale}
    \mathcal{M}^\beta_{T-t}\coloneqq\int_0^t\nabla\mathcal{R}^\beta_{T-t}(X_{T-t})\,d\overline{W}^{\mathbb{P}^\beta}_{T-t},\quad\text{for all}\quad0\leq t\leq T.
\end{equation}
Then the semimartingale decomposition of $(\mathcal{R}^\beta_{T-t})_{0\leq t\leq T}$ can be written as $\mathcal{R}^\beta_{T-t}-\mathcal{R}^\beta_{T}=\mathcal{M}^\beta_{T-t}+\mathcal{F}^\beta_{T-t}$, with $0\leq t\leq T.$ It is remarkable that in the absence of perturbation, taking $\mathbb{P}^0$-expectation to the cumulative Fisher information process (\ref{cumulative Fisher information process}) gives us the cumulative integral of the Fisher information (\ref{Fisher information}) modulo a multiplicative factor $\frac{1}{2}$, i.e.~
\begin{equation*}
    \mathbb{E}^{\mathbb{P}^0}\big[\mathcal{F}^0_{t}\big]=\frac{1}{2}\int_t^T\mathbb{E}^{\mathbb{P}^0}\big[\normy{\nabla\mathcal{R}^0_{\theta}(X_{\theta})}^2\big]\,d\theta=\frac{1}{2}\int_t^T\mathbb{I}\big[P^0_{\theta}|Q\big]\,d\theta,\quad\text{for all}\quad0\leq t\leq T.
\end{equation*}
Now, we verify that $(\mathcal{M}^\beta_{T-t})_{0\leq t\leq T}$ is an uniformly integrable martingale.
\begin{lemma}\label{lem: M is L2 martingale}
    The backward-time continuous local martingale $(\mathcal{M}^\beta_{T-t})_{0\leq t\leq T}$ is a square integrable martingale adapted to the filtration $(\mathcal{G}^\beta_{T-t})_{0\leq t\leq T}$.
\end{lemma}
\begin{proof}
    It is sufficient to show that $(\mathcal{M}^\beta_{T-t})_{0\leq t\leq T}$ is bounded in $L^2(\mathbb{P}^\beta)$. Since we have assumed the continuity of $t\mapsto\nabla\log\ell^\beta_t(x)$ on $[0,T]$, for any fixed $x\in\mathbb{R}^d$, and by the continuity of the sample paths of $(X_t)_{0\leq t\leq T}$, we observe that
    \begin{equation*}
        \int_0^{T-\epsilon}\normy{\nabla\log\ell^\beta_{T-\theta}(X_{T-\theta})}^2\,d\theta<\infty,\quad\mathbb{P}^\beta\text{-a.s.}\quad\text{for all}\quad0<\epsilon<T.
    \end{equation*}
    On this account, define the sequence of stopping times by
    \begin{equation*}
        \tau^\beta_k\coloneqq T\wedge\inf\big\{t\geq0:\,\int_0^t\normy{\nabla\log\ell^\beta_{T-\theta}(X_{T-\theta})}^2\,d\theta\geq k\big\},\quad\text{for all}\quad k\in\mathbb{N}.
    \end{equation*}
    Then $(\tau^\beta_k)_{k\in\mathbb{N}}$ is non-decreasing and converges to $T$, $\mathbb{P}^\beta$-a.s. Henceforth, $(\tau^\beta_k)_{k\in\mathbb{N}}$ is a localization sequence for the local martingale $(\mathcal{M}^\beta_{T-t})_{0\leq t\leq T}$. The stopped process $(\mathcal{M}^\beta_{T-\tau^\beta_k\wedge t})_{0\leq t\leq T}$ is therefore a $L^2(\mathbb{P}^\beta)$-bounded martingale adapted to $(\mathcal{G}^\beta_{T-t})_{0\leq t\leq T}$, for each $k\in\mathbb{N}$.\par
    Taking $\mathbb{P}^\beta$-expectation to the process $(\mathcal{R}^\beta_{T-t})_{0\leq t\leq T}$ at the stopping time $\tau^\beta_k$, we observe that
    \begin{equation*}
        \frac{1}{2}\mathbb{E}^{\mathbb{P}^\beta}\big[\int_0^{\tau^\beta_k}\normy{\nabla\mathcal{R}^\beta_{T-\theta}(X_{T-\theta})}^2\,d\theta\big]+\mathbb{E}^{\mathbb{P}^\beta}\big[\int_0^{\tau^\beta_k}\big(2\beta\cdot\nabla\psi-\sum\limits_{i=1}^d\frac{\partial\beta^{(i)}}{\partial x_i}\big)(X_{T-\theta})\,d\theta\big]=\mathbb{H}\big[P^\beta_{T-\tau^\beta_k}|Q\big]-\mathbb{H}\big[P^\beta_{T}|Q\big],
    \end{equation*}
    for each $k\in\mathbb{N}$.
    Since the perturbation field $\beta(\cdot)=\nabla B(\cdot)$ is of class $\mathcal{C}^\infty(\mathbb{R}^d;\mathbb{R}^d)$ with compact support,
    \begin{equation*}
        C_1\coloneqq\mathbb{E}^{\mathbb{P}^\beta}\big[\int_0^T\big|2\beta\cdot\nabla\psi-\sum\limits_{i=1}^d\frac{\partial\beta^{(i)}}{\partial x_i}\big|(X_{T-\theta})\,d\theta\big]<\infty.
    \end{equation*}
    Henceforth,
    \[
        \frac{1}{2}\mathbb{E}^{\mathbb{P}^\beta}\big[\int_0^{\tau^\beta_k}\normy{\nabla\mathcal{R}^\beta_{T-\theta}(X_{T-\theta})}^2\,d\theta\big]\leq C_1+\mathbb{H}\big[P^\beta_{T-\tau^\beta_k}|Q\big]-\mathbb{H}\big[P^\beta_{T}|Q\big].
    \]\par
    Only in this proof, for all $0\leq t\leq T$, we denote by $Q^\beta_t$ the $\sigma$-finite Borel measure on $\mathbb{R}^d$ with density $\exp(-2(\psi+BI_{\{t>t_0\}})(\cdot))$ against the Lebesgue measure on $\mathbb{R}^d$. A variant argument to Lemma \ref{relative entropy H is Q-submartingale} implies,
    \begin{equation*}
        \mathbb{H}\big[P^\beta_{T-\tau^\beta_k}|Q^\beta_{T-\tau^\beta_k}\big]\leq\mathbb{H}\big[P^\beta_{0}|Q^\beta_0\big]\quad\text{for all}\quad k\in\mathbb{N}.
    \end{equation*}
    Now that we have the $\mathbb{P}^\beta$-a.s. boundedness of $k\mapsto\mathbb{H}[P^\beta_{T-\tau^\beta_k}|Q^\beta_{T-\tau^\beta_k}]$. To proceed with an estimate of the terms $\mathbb{H}[P^\beta_{T-\tau^\beta_k}|Q]$ and $\mathbb{H}[P^\beta_{T}|Q]$, we observe that
    \begin{equation*}
        C_2\coloneqq2\max\big\{\abs{B(x)}:\,x\in\mathbb{R}^d\big\}<\infty.
    \end{equation*}
    It is immediate that $\mathbb{H}[P^\beta_{0}|Q^\beta_{0}]\leq\mathbb{H}[P_0|Q]+C_2$ and 
    \begin{equation*}
        \mathbb{H}\big[P^\beta_{T-\tau^\beta_k}|Q\big]-C_2\leq\mathbb{H}\big[P^\beta_{T-\tau^\beta_k}|Q^\beta_{T-\tau^\beta_k}\big]\leq\mathbb{H}\big[P^\beta_{T-\tau^\beta_k}|Q\big]+C_2,\quad\text{for each}\quad k\in\mathbb{N}.
    \end{equation*}
    In consequence,
    \begin{equation*}
        \mathbb{H}\big[P^\beta_{T-\tau^\beta_k}|Q\big]\leq\mathbb{H}\big[P^\beta_{T-\tau^\beta_k}|Q^\beta_{T-\tau^\beta_k}\big]+C_2\leq\mathbb{H}\big[P^\beta_{0}|Q^\beta_{0}\big]+C_2 \leq\mathbb{H}\big[P_0|Q\big]+2C_2\quad\text{for each}\quad k\in\mathbb{N}
    \end{equation*}
    as well as 
    \begin{equation*}
        \frac{1}{2}\mathbb{E}^{\mathbb{P}^\beta}\big[\int_0^{\tau^\beta_k}\normy{\nabla\mathcal{R}^\beta_{T-\theta}(X_{T-\theta})}^2\,d\theta\big]\leq C_1+\mathbb{H}\big[P_0|Q\big]+2C_2-\mathbb{H}\big[P^\beta_T|Q\big]<\infty
    \end{equation*}
    for all $k\in\mathbb{N}$. Since $\tau^\beta_k\to T$ as $k\to\infty$ $\mathbb{P}^\beta$-a.s., the monotone convergence theorem yields
    \begin{equation}\label{monotone convergence}
        \mathbb{E}^{\mathbb{P}^\beta}\big[\big\langle\mathcal{M}^\beta\big\rangle_0\big]=\mathbb{E}^{\mathbb{P}^\beta}\big[\int_0^T\normy{\nabla\mathcal{R}^\beta_{T-\theta}(X_{T-\theta})}^2\,d\theta\big]\leq2\big(C_1+\mathbb{H}\big[P_0|Q\big]+2C_2-\mathbb{H}\big[P^\beta_T|Q\big]\big)<\infty.
    \end{equation}
    Henceforth, the continuous local martingale $(\mathcal{M}^\beta_{T-t})_{0\leq t\leq T}$ is a $(\mathcal{G}^\beta_{T-t})_{0\leq t\leq T}$-martingale bounded in $L^2(\mathbb{P}^\beta)$. And the assertion is verified.
\end{proof}
The martingale property of $(\mathcal{M}^\beta_{T-t})_{0\leq t\leq T}$ from the semimartingale decomposition of $(\mathcal{R}^\beta_{T-t})_{0\leq t\leq T}$ allows us to transfer our conclusions of the pathwise behavior of $(\mathcal{R}^\beta_{T-t})_{0\leq t\leq T}$ to the quantity $\mathbb{H}^\beta[P^\beta_{T-t}|Q]$ with $0\leq t\leq T$ by taking $\mathbb{P}^\beta$-expectation, without any localization sequence of stopping times. In fact, Lemma \ref{lem: M is L2 martingale} guarantees that our trajectorial formulation can retrieve the known results on relative entropy dissipation.


\section{Applications to relative entropy dissipation}\label{sec: trajectorial dissipation}
This section contains the main results of this expository article, namely, the trajectorial formulation of the relative entropy dissipation. Based on the preliminaries in Sections \ref{sec: time-reversal} and \ref{sec: semimartingale}, our results describe some remarkable features, for instance Theorems \ref{perturbed trajectorial entropy decay, time-displacement} and \ref{perturbed trajectorial entropy decay, derivative1}, of the pathwise behavior of $\mathcal{R}^\beta$ under time-reversal. The reason why we look at things backward in time is explained at the end of this section, where the less transparent forward-time approach is compared to our derivations in Section \ref{sec: semimartingale}.\par
The trajectorial approach is an advancement towards understanding the random fluctuations of the relative entropy of a complex system \cite{Breunung/Kogelbauer/Haller, Thurner/Corominas-Murtra/Hanel}. This approach reveals more information from the Itô-Langevin stochastic system than the known classical results. Indeed, taking $\mathbb{P}^\beta$-expectation retrieves the dynamics of relative entropy. Let us now focus on this trajectorial interpretation. 
\subsection{Time-displacement and derivative of $\mathcal{R}^\beta$}
The trajectorial interpretation to the dissipation of relative entropy is referred to Theorems \ref{thm: Trajectorial rate of relative entropy dissipation, time-displacement, perturbed} and \ref{thm: Trajectorial rate of relative entropy dissipation, derivative} when the Itô-Langevin dynamics is placed under a perturbation $\beta(\cdot):\mathbb{R}^d\to\mathbb{R}^d$ initiated at $t_0\geq0$. Corollaries \ref{cor: 5.4} and \ref{cor: Trajectorial rate of relative entropy dissipation, unperturbed} record the scenario in the absence of perturbation. As a prelude, these corollaries will be of independent interest for the understanding of the so-called \textit{steepest descent property}, to be discussed in Section \ref{sec: wasserstein}. Our first step is an argument on the regularity control of the ratio between $\ell^\beta$ and $\ell^0$.
\begin{lemma}\label{lem: 5.1}
    Fix the time interval $[0,T]$ with $T>t_0$. There exist $C_1,C_2>0$ such that
    \begin{equation}\label{bdd, l ratio}
        C_1\leq\frac{\ell^\beta_t(x)}{\ell^0_t(x)}=\frac{p_t^\beta(x)}{p^0_t(x)}\leq C_2,\quad\text{for all}\quad(t,x)\in[0,T]\times\mathbb{R}^d.
    \end{equation}
\end{lemma}
\begin{proof}
    In the Itô-Langevin dynamics (\ref{perturbed Itô-Langevin dynamics}), the forward-time coordinate process is denoted by $(X_t)_{0\leq t\leq T}$. We have also denoted by $(W^\beta_t)_{0\leq t\leq T}$ the $d$-dimensional $(\mathcal{F}_t)_{0\leq t\leq T}$-Brownian motion under $\mathbb{P}^\beta$, and by $(W^0_t)_{0\leq t\leq T}$ the Brownian motion under $\mathbb{P}^0$. Hence,
    \begin{equation*}
        W^0_t-W^0_{t_0}=W^\beta_t-W^\beta_{t_0}-\int_{t_0}^t\beta(X_\theta)\,d\theta,\quad\text{for all}\quad t_0\leq t\leq T.
    \end{equation*}
    By the Girsanov theorem \cite[Proposition 5.21]{Le Gall}, the density between $\mathbb{P}^\beta$ and $\mathbb{P}^0$ amounts to,
    \begin{equation}
        \label{Z beta: control 1}
        Z^\beta_t\coloneqq\frac{\mathbb{P}^\beta}{\mathbb{P}^0}\bigg|_{\mathcal{F}_t}=\exp\big(-\int_{t_0}^t\beta(X_\theta)\,dW^0_\theta-\frac{1}{2}\int_{t_0}^t\normy{\beta(X_\theta)}^2\,d\theta\big),\quad\text{for all}\quad t_0\leq t\leq T.
    \end{equation}
    Notice that for each $(t,x)\in[t_0,T]\times\mathbb{R}^d$, the ratio $\ell^\beta_t(x)/\ell^0_t(x)=p^\beta_t(x)/p^0_t(x)$ is equal to $Z^\beta_t$, under the condition $X_t=x$, i.e.
    \begin{equation*}
        \frac{\ell^\beta_t(x)}{\ell^0_t(x)}=\mathbb{E}^{\mathbb{P}^0}\big[Z^\beta_t|X_t=x\big],\quad\text{for all}\quad(t,x)\in[0,T]\times\mathbb{R}^d.
    \end{equation*}
    Therefore, if we manage to uniformly bound the logarithm $(\log Z^\beta_t)_{0\leq t\leq T}$, then the uniform boundedness of $|\ell^\beta_t(x)/\ell^0_t(x)|$ follows. Since the perturbation $\beta(\cdot):\mathbb{R}^d\to\mathbb{R}^d$ is smooth with compact support,
    \begin{equation*}
        \frac{1}{2}\int_{t_0}^t\normy{\beta(X_\theta)}^2\,d\theta\leq C^\prime,\quad\text{for all}\quad t_0\leq t\leq T,\quad\mathbb{P}^\beta\text{-a.s.}
    \end{equation*}
    for some constant $C^\prime>0$. Since $\beta(\cdot)$ is of gradient type, i.e.~$\beta(\cdot)=\nabla B(\cdot)$ with $B(\cdot)$ of class $\mathcal{C}^\infty(\mathbb{R}^d;\mathbb{R})$ and compactly supported, then the Itô formula gives
    \begin{equation}
        \label{Z beta: control 2}
        \int_{t_0}^t\beta(X_\theta)\,dW^0_\theta=B(X_t)-B(X_{t_0})+\frac{1}{2}\int_{t_0}^t\big(2\beta\cdot\nabla\psi-\sum\limits_{i=1}^d\frac{\partial\beta^{(i)}}{\partial x_i}\big)(X_\theta)\,d\theta,\quad\text{for all}\quad t_0\leq t\leq T.
    \end{equation}
    Invoking the compact supportness of $B(\cdot)$ again, it is then obvious that
    \begin{equation*}
        \big|\int_{t_0}^t\beta(X_\theta)\,dW^0_\theta\big|\leq C^{\prime\prime},\quad\text{for all}\quad t_0\leq t\leq T,\quad\mathbb{P}^\beta\text{-a.s.}
    \end{equation*}
    for some constant $C^{\prime\prime}>0$. And this implies that $|\log Z^\beta_t|\leq C^\prime+C^{\prime\prime}$ for all $t_0\leq t\leq T$ $\mathbb{P}^\beta$-a.s., whence the assertion is verified.
\end{proof}
The framework of our discussion on the pathwise behavior of $\mathcal{R}^\beta$ is based on a time-reversal perspective. The following Theorems \ref{thm: Trajectorial rate of relative entropy dissipation, time-displacement, perturbed} and \ref{thm: Trajectorial rate of relative entropy dissipation, derivative} present the displacement and time-derivative of $\mathcal{R}^\beta$ backward in time. Lemma \ref{lem: 5.1} provides quantitative control on the deviation effect of the perturbation $\beta(\cdot)$ based on its smoothness and compact support, which is necessary to our further derivation on the trajectorial dynamics of the relative entropy process $\mathcal{R}^\beta$. 
\begin{theorem}\label{thm: Trajectorial rate of relative entropy dissipation, time-displacement, perturbed}
    Fix the time interval $[0,T]$ with $T>t_0$. The time-reversal of the relative entropy process $\mathcal{R}^\beta$ satisfies, for all $0\leq t< T-t_0$, the following $\mathbb{P}^\beta$-a.s. trajectorial relation,
    \begin{equation}\begin{aligned}\label{perturbed trajectorial entropy decay, time-displacement}
        \mathbb{E}^{\mathbb{P}^\beta}\big[\mathcal{R}^\beta_{t_0}(X_{t_0})|\mathcal{G}_{T-t}\big]-\mathcal{R}^\beta_{T-t}
        (X_{T-t})&=\mathbb{E}^{\mathbb{P}^\beta}\big[\int_t^{T-t_0}\big(2\beta\cdot\nabla\psi-\sum\limits_{i=1}^d\frac{\partial\beta^{(i)}}{\partial x_i}\big)(X_{T-\theta}) \,d\theta|\mathcal{G}_{T-t}\big]\\
        &\quad+\frac{1}{2}\mathbb{E}^{\mathbb{P}^\beta}\big[\int_t^{T-t_0}\normy{\nabla\mathcal{R}^\beta_{T-\theta}(X_{T-\theta})}^2 \,d\theta|\mathcal{G}_{T-t}\big].
    \end{aligned}\end{equation}
\end{theorem}
\begin{proof}
    Applying Lemma (\ref{lem: M is L2 martingale}), the local martingale part from the semimartingale decomposition (\ref{semimartingale decomposition of R}) of $\mathcal{R}^\beta$ is a square integrable martingale. Hence, taking $\mathcal{G}_{T-t}$-conditional expectation with respect to $\mathbb{P}^\beta$ on (\ref{semimartingale decomposition of R}) cancels this martingale part. And the assertion follows.
\end{proof}
The perturbation terms in (\ref{perturbed trajectorial entropy decay, time-displacement}) clouds the implication of the phrase \textit{dissipation}. Nonetheless, this term indicates how the perturbation $\beta(\cdot)$ entangles with the potential $\psi(\cdot)$, and henceforth affects the Itô-Langevin stochastic dynamics (\ref{perturbed Itô-Langevin dynamics}). But if we collapse the perturbation $\beta(\cdot)$, things become more transparent.
\begin{corollary}\label{cor: 5.4}
    Switching off the perturbation $\beta(\cdot):\mathbb{R}^d\to\mathbb{R}^d$, for all $0\leq t< T-t_0$, Theorem \ref{thm: Trajectorial rate of relative entropy dissipation, time-displacement, perturbed} reduces to the $\mathbb{P}^0$-a.s. trajectorial displacement of relative entropy dissipation,
    \begin{equation}\label{unperturbed trajectorial entropy decay, time-displacement}
        \mathbb{E}^{\mathbb{P}^0}\big[\mathcal{R}^0_{t_0}(X_{t_0})|\mathcal{G}_{T-t}\big]-\mathcal{R}^0_{T-t}
        (X_{T-t})=\frac{1}{2}\mathbb{E}^{\mathbb{P}^0}\big[\int_t^{T-t_0}\normy{\nabla\mathcal{R}^0_{T-\theta}(X_{T-\theta})}^2 \,d\theta|\mathcal{G}_{T-t}\big].
    \end{equation}
\end{corollary}
Remember that he backward cumulative Fisher information process is defined in (\ref{cumulative Fisher information process}), which is exactly the integrand in (\ref{unperturbed trajectorial entropy decay, time-displacement}). If we present things forward in time, for instance replacing $T-t$ by $T-(T-t)$, then we observe that the relative entropy process $(\mathcal{R}^\beta_t(X_t))_{t\geq0}$ is monotonically decreasing along all of its trajectories $\mathbb{P}^0$-a.s., conforming to the phrase \textit{dissipation}.\par
On the other hand, the Fisher information quantity (\ref{Fisher information}) appears again if we take the time-derivative of the relative entropy displacement (\ref{perturbed trajectorial entropy decay, time-displacement}). This can be seen as the trajectorial rate of time-evolution of $\mathcal{R}^\beta$. However, things become delicate if we explore more on such limiting trajectorial behavior.\par
Lemma $\ref{lem: 5.1}$ provides sufficient regularity to allow us taking the limit $t\nearrow T-t_0$ in (\ref{perturbed trajectorial entropy decay, time-displacement}) divided by $T-t_0-t$, which gives us the time-derivative of backward process $\mathcal{R}^\beta$ under time-reversal. This differential structure of the trajectorial relative entropy process will eventually shed light on the know classical results on entropy dissipation which have been displayed in Section \ref{sec: stochastic dynamics}.
\begin{theorem}\label{thm: Trajectorial rate of relative entropy dissipation, derivative}
    Fix the time interval $[0,T]$ with $T>t_0$. The backward relative entropy process $\mathcal{R}^\beta$ satisfies the limiting trajectorial identity,
    \begin{equation}
        \label{perturbed trajectorial entropy decay, derivative1}
        \lim\limits_{t\nearrow T-t_0}\frac{1}{T-t_0-t}\bigg(\mathbb{E}^{\mathbb{P}^\beta}\big[\mathcal{R}^\beta_{t_0}(X_{t_0})|\mathcal{G}_{T-t}\big]-\mathcal{R}^\beta_{T-t}
        (X_{T-t})\bigg)=\frac{1}{2}\normy{\nabla\mathcal{R}^0_{t_0}(X_{t_0})}^2+\big(2\beta\cdot\nabla\psi-\sum\limits_{i=1}^d\frac{\partial\beta^{(i)}}{\partial x_i}\big)(X_{t_0}),
    \end{equation}
    where the limit in (\ref{perturbed trajectorial entropy decay, derivative1}) is taken in $L^1(\mathbb{P}^\beta)$.
\end{theorem}
\begin{proof}
    Viewing Theorem \ref{thm: Trajectorial rate of relative entropy dissipation, time-displacement, perturbed} and invoking the dominated convergence theorem, we deduce from the smoothness and compact supportness of $\beta(\cdot)$, together with the uniform boundedness of Lemma \ref{lem: 5.1}, that
    \begin{equation*}
        \lim\limits_{t\nearrow T-t_0}\frac{1}{T-t_0-t}\mathbb{E}^{\mathbb{P}^\beta}\big[\big|\mathcal{R}^\beta_{T-t}(X_{T-t})-\mathcal{R}^\beta_{t_0}(X_{t_0})\big|\big]=\mathbb{E}^{\mathbb{P}^\beta}\big[\big|\frac{1}{2}\normy{\nabla\mathcal{R}^0_{t_0}(X_{t_0})}^2+\big(2\beta\cdot\nabla\psi-\sum\limits_{i=1}^d\frac{\partial\beta^{(i)}}{\partial x_i}\big)(X_{t_0})\big|\big].
    \end{equation*}
    The fundamental theorem of calculus for continuous functions implies that
    \begin{equation*}
        \lim\limits_{t\nearrow T-t_0}\frac{1}{T-t_0-t}\big(\mathcal{R}^\beta_{T-t}(X_{T-t})-\mathcal{R}^\beta_{t_0}(X_{t_0})\big)=\frac{1}{2}\normy{\nabla\mathcal{R}^0_{t_0}(X_{t_0})}^2+\big(2\beta\cdot\nabla\psi-\sum\limits_{i=1}^d\frac{\partial\beta^{(i)}}{\partial x_i}\big)(X_{t_0}),\quad\mathbb{P}^\beta\text{-a.s.}
    \end{equation*}
   The Scheffé lemma \cite[Section 5.10]{Williams} says that if a sequence of integrable random variables converges to a limiting random variable almost surely, then the convergence in $L^1$ is equivalent to the convergence of their $L^1$ norms. Using the Scheffé lemma,
   \begin{equation}\label{Scheffé}
       \lim\limits_{t\nearrow T-t_0}\mathbb{E}^{\mathbb{P}^\beta}\bigg[\bigg|\frac{\mathcal{R}^\beta_{T-t}(X_{T-t})-\mathcal{R}^\beta_{t_0}(X_{t_0})}{T-t_0-t}-\frac{1}{2}\normy{\nabla\mathcal{R}^0_{t_0}(X_{t_0})}^2-\big(2\beta\cdot\nabla\psi-\sum\limits_{i=1}^d\frac{\partial\beta^{(i)}}{\partial x_i}\big)(X_{t_0})\bigg|\bigg]=0.
   \end{equation}
   Moreover, the triangle inequality and the Jensen inequality for conditional expectation yields 
   \begin{equation*}\begin{aligned}
       &\mathbb{E}^{\mathbb{P}^\beta}\bigg[\bigg|\frac{\mathbb{E}^{\mathbb{P}^\beta}[\mathcal{R}^\beta_{t_0}(X_{t_0})|\mathcal{G}_{T-t}]-\mathcal{R}^\beta_{T-t}
        (X_{T-t})}{T-t_0-t}-\frac{1}{2}\normy{\nabla\mathcal{R}^0_{t_0}(X_{t_0})}^2-\big(2\beta\cdot\nabla\psi-\sum\limits_{i=1}^d\frac{\partial\beta^{(i)}}{\partial x_i}\big)(X_{t_0})\bigg|\bigg]\\
        &\quad\leq\mathscr{A}^\beta_{t_0,T-t}+\mathscr{B}^\beta_{t_0,T-t}+\mathscr{C}^\beta_{t_0,T-t},
   \end{aligned}\end{equation*}
   where 
   \begin{equation*}\begin{aligned}
       \mathscr{A}^\beta_{t_0,T-t}&\coloneqq\mathbb{E}^{\mathbb{P}^\beta}\bigg[\bigg|\frac{\mathcal{R}^\beta_{T-t}(X_{T-t})-\mathcal{R}^\beta_{t_0}(X_{t_0})}{T-t_0-t}-\frac{1}{2}\normy{\nabla\mathcal{R}^0_{t_0}(X_{t_0})}^2-\big(2\beta\cdot\nabla\psi-\sum\limits_{i=1}^d\frac{\partial\beta^{(i)}}{\partial x_i}\big)(X_{t_0})\bigg|\bigg],\\
       \mathscr{B}^\beta_{t_0,T-t}&\coloneqq\mathbb{E}^{\mathbb{P}^\beta}\bigg[\bigg|\mathbb{E}^{\mathbb{P}^\beta}\big[\big(2\beta\cdot\nabla\psi-\sum\limits_{i=1}^d\frac{\partial\beta^{(i)}}{\partial x_i}\big)(X_{t_0})|\mathcal{G}_{T-t}\big]-\big(2\beta\cdot\nabla\psi-\sum\limits_{i=1}^d\frac{\partial\beta^{(i)}}{\partial x_i}\big)(X_{t_0})\bigg|\bigg],\\
       \mathscr{C}^\beta_{t_0,T-t}&\coloneqq\mathbb{E}^{\mathbb{P}^\beta}\bigg[\bigg|\mathbb{E}^{\mathbb{P}^\beta}\big[\frac{1}{2}\normy{\nabla\mathcal{R}^0_{t_0}(X_{t_0})}^2|\mathcal{G}_{T-t}\big]-\frac{1}{2}\normy{\nabla\mathcal{R}^0_{t_0}(X_{t_0})}^2\bigg|\bigg].
   \end{aligned}\end{equation*}
   By (\ref{Scheffé}), we know $\mathscr{A}^\beta_{t_0,T-t}\to0$ as $t\nearrow T-t_0$. Using \cite[Theorem 9.4.8]{Chung} and the right-continuity of $(\mathcal{G}_{T-t})_{0\leq t\leq T}$, we know $\mathscr{B}^\beta_{t_0,T-t}\to0$ and $\mathscr{C}^\beta_{t_0,T-t}\to0$ as $t\nearrow T-t_0$. Combining these facts, the assertion (\ref{perturbed trajectorial entropy decay, derivative1}) is verified.
\end{proof}
The limiting identity (\ref{perturbed trajectorial entropy decay, derivative1}) on the time-derivative of the process $\mathcal{R}^\beta$ indicates that this time-derivative is split into two parts: the $\mathbb{P}^\beta$-integrand of the Fisher information and a perturbation term induced by $\beta(\cdot)$. This expression conforms with the spirit that the Itô-Langevin system is perturbed. And this expression is transparent in the sense that the perturbation term is separate from the Fisher information integrand. In consequence, this perturbation term vanishes when $\beta\equiv0$.
\begin{corollary}\label{cor: Trajectorial rate of relative entropy dissipation, unperturbed}
    Switching off the perturbation $\beta(\cdot):\mathbb{R}^d\to\mathbb{R}^d$, Theorem \ref{thm: Trajectorial rate of relative entropy dissipation, derivative} reduces to the unperturbed limiting trajectorial identity
    \begin{equation}
        \label{unperturbed trajectorial entropy decay}
        \lim\limits_{t\nearrow T-t_0}\frac{1}{T-t_0-t}\bigg(\mathbb{E}^{\mathbb{P}^0}\big[\mathcal{R}^0_{t_0}(X_{t_0})|\mathcal{G}_{T-t}\big]-\mathcal{R}^0_{T-t}(X_{T-t})\bigg)=\frac{1}{2}\norm{\nabla\mathcal{R}^0_{t_0}(X_{t_0})}^2,
    \end{equation}
    where the limit in (\ref{unperturbed trajectorial entropy decay}) is taken in $L^1(\mathbb{P}^0)$.
\end{corollary}
Theorem \ref{thm: Trajectorial rate of relative entropy dissipation, derivative} and Corollary \ref{cor: Trajectorial rate of relative entropy dissipation, unperturbed} present time-derivatives of the relative entropy from a trajectorial approach, in the perturbed and unperturbed cases, respectively. In the subsequent paragraphs, we will see how these trajectorial identities retrieve the known classical phenomena on the dissipation of relative entropy.

\subsection{Consequences on the classical results}
The identities presented in Theorems \ref{thm: Trajectorial rate of relative entropy dissipation, time-displacement, perturbed} and \ref{thm: Trajectorial rate of relative entropy dissipation, derivative} reveal the trajectorial dynamics of the relative entropy process $\mathcal{R}^\beta$. For the computational convenience, these results are written backward in time. Nevertheless, after taking $\mathbb{P}^\beta$-expectation, these results conform to the known phenomena on the forward-time relative entropy dissipation. And this consequence confirms that the trajectorial identities presented in this expository article yield more information than the classical approach on the relative entropy dissipation.\par
The following paragraphs show how the trajectorial approach eventually rediscovers the known results on the relative entropy dissipation of the Itô-Langevin stochastic system. To begin with, Lemma \ref{lem: 5.6} gives another deviation control on the perturbation effect of $\beta(\cdot)$, apart from Lemma \ref{lem: 5.1}. Such control is important when we discuss the time-derivatives of relative entropy quantity.
\begin{lemma}\label{lem: 5.6}
    Fix the time interval $[0,T]$ with $T>t_0$. There exist $C_1,C_2>0$ such that
    \begin{equation}\label{ineq 1, lem 5.2}
        \big|\frac{\ell^\beta_{T-t}(x)}{\ell^0_{T-t}(x)}-1\big|\leq C_1(T-t_0-t)
    \end{equation}
    as well as
    \begin{equation}\label{ineq 2, lem 5.2}
        \mathbb{E}^{\mathbb{P}^0}\big[\int_t^{T-t_0}\normy{\nabla(\mathcal{R}^\beta_{T-\theta}-\mathcal{R}^0_{T-\theta})(X_{T-\theta})}^2\,d\theta\big]\leq C_2(T-t_0-t)^2,
    \end{equation}
    for all $0\leq t< T-t_0$ and $x\in\mathbb{R}^d$.
\end{lemma}
\begin{proof}
    Only in this proof, we will denote by $L^\beta_t(x)\coloneqq\ell^\beta_t(x)/\ell^0_t(x)$ for all $(t,x)\in[0,T]\times\mathbb{R}^d$. Since for all $t\geq0$, $\log L^\beta_t(X_t)=\mathcal{R}^\beta_t(X_t)-\mathcal{R}^0_t(X_t)$, then
    \begin{equation*}\begin{aligned}
        &\mathbb{E}^{\mathbb{P}^0}\big[\int_0^{T-t_0}\normy{\nabla\log L^\beta_{T-\theta}(X_{T-\theta})}^2\,d\theta\big]\\
        &\quad\leq2\mathbb{E}^{\mathbb{P}^0}\big[\int_0^{T-t_0}\normy{\nabla\mathcal{R}^0_{T-\theta}(X_{T-\theta})}^2\,d\theta\big]+2\mathbb{E}^{\mathbb{P}^0}\big[\int_0^{T-t_0}\normy{\nabla\mathcal{R}^\beta_{T-\theta}(X_{T-\theta})}^2\,d\theta\big]\\
        &\quad\leq2\mathbb{E}^{\mathbb{P}^0}\big[\int_0^{T-t_0}\normy{\nabla\mathcal{R}^0_{T-\theta}(X_{T-\theta})}^2\,d\theta\big]+2\mathbb{E}^{\mathbb{P}^\beta}\big[\int_0^{T-t_0}Z^\beta_{T-\theta}\normy{\nabla\mathcal{R}^\beta_{T-\theta}(X_{T-\theta})}^2\,d\theta\big].
    \end{aligned}\end{equation*}
    Using (\ref{Z beta: control 1}) and (\ref{Z beta: control 2}), there exists $C^\prime>0$ such that $Z^\beta_t(x)\leq C^\prime$ for all $(t,x)\in[0,T]\times\mathbb{R}^d$. Hence,
    \begin{equation*}\begin{aligned}
        &\mathbb{E}^{\mathbb{P}^0}\big[\int_0^{T-t_0}\normy{\nabla\log L^\beta_{T-\theta}(X_{T-\theta})}^2\,d\theta\big]\\
        &\quad\leq2\mathbb{E}^{\mathbb{P}^0}\big[\int_0^{T-t_0}\normy{\nabla\mathcal{R}^0_{T-\theta}(X_{T-\theta})}^2\,d\theta\big]+2C^\prime\mathbb{E}^{\mathbb{P}^\beta}\big[\int_0^{T-t_0}\normy{\nabla\mathcal{R}^\beta_{T-\theta}(X_{T-\theta})}^2\,d\theta\big]\\
        &\quad\leq2C^{\prime\prime}(1+C^\prime)<\infty,
    \end{aligned}\end{equation*}
    for some $C^{\prime\prime}>0$, where the last inequality is due to (\ref{monotone convergence}). Since $\nabla\log L^\beta_t(x)=\nabla L^\beta_t(x)/L^\beta_t(x)$ for all $x\in\mathbb{R}^d$, using (\ref{bdd, l ratio}) again, we have
    \begin{equation}\label{martingale part is 2nd-integrable}
        \mathbb{E}^{\mathbb{P}^0}\big[\int_0^{T-t_0}\normy{\nabla L^\beta_{T-\theta}(X_{T-\theta})}^2\,d\theta\big]\leq C^{\prime\prime\prime}
    \end{equation}
    for some constant $C^{\prime\prime\prime}>0$. Using (\ref{semimartingale decomposition of l}) for both the perturbed and unperturbed cases and by the Itô formula, we see that the time-reversal $(L^\beta_{T-t}(X_{T-t}))_{0\leq t\leq T-t_0}$ satisfies
    \begin{equation}\label{SDE for L}
        dL^\beta_{T-t}(X_{T-t})=\nabla L^\beta_{T-t}(X_{T-t})\,d\overline{W}^{\mathbb{P}^0}_{T-t}-\big(L^\beta_{T-t}\sum\limits_{i=1}^d\frac{\partial\beta^{(i)}}{\partial x_i}\big)(X_{T-t})-L^\beta_{T-t}\big(\beta\cdot\nabla\log\text{(}p^0_{T-t}L^\beta_{T-t})\big)(X_{T-t})\,dt
    \end{equation}
    for all $0\leq t<T-t_0$, with respect to the filtration $(\mathcal{G}_{T-t})_{0\leq t\leq T-t_0}$. In view of (\ref{martingale part is 2nd-integrable}), the martingale part from the semimartingale decomposition of $(L^\beta_{T-t}(X_{T-t}))_{0\leq t\leq T-t_0}$ is $L^2(\mathbb{P}^0)$-bounded. Its drift term vanishes when $X_{T-t}$ exits the compact support of $\beta(\cdot)$ in $\mathbb{R}^d$. Hence, the drift term is bounded by
    \begin{equation*}
        C^{\prime\prime\prime\prime}\coloneqq\sup\big\{\big|L^\beta_{T-t}\sum\limits_{i=1}^d\frac{\partial\beta^{(i)}}{\partial x_i}\big|(x)+\big|L^\beta_{T-t}\big(\beta\cdot\nabla\log\text{(}p^0_{T-t}L^\beta_{T-t})\big)\big|(x):\,x\in\mathbb{R}^d,\,0\leq t\leq T-t_0\big\}<\infty.
    \end{equation*}
    Henceforth, the processes
    \begin{equation*}
        L^\beta_{T-t}(X_{T-t})+C^{\prime\prime\prime\prime}t\qquad\text{and}\qquad L^\beta_{T-t}(X_{T-t})-C^{\prime\prime\prime\prime}t,\quad\text{with}\quad 0\leq t\leq T-t_0
    \end{equation*}
    are respectively $(\mathcal{G}_{T-t})_{0\leq t\leq T-t_0}$-submartingale and $(\mathcal{G}_{T-t})_{0\leq t\leq T-t_0}$-supermartingale. Then,
    \begin{equation*}
        \big|\mathbb{E}^{\mathbb{P}^0}\big[L^\beta_{t_0}(X_{t_0})|X_{T-t}=x\big]- L^\beta_{T-t}(x)\big|\leq C^{\prime\prime\prime\prime}(T-t_0-t),\quad\text{for all}\quad(t,x)\in[0,T-t_0]\times\mathbb{R}^d.
    \end{equation*}
    Since $L^\beta_{t_0}(\cdot)\equiv1$, taking $C_1\coloneqq C^{\prime\prime\prime\prime}$, (\ref{ineq 1, lem 5.2}) is verified. Using (\ref{SDE for L}) and the Itô formula, we observe
    \begin{equation*}\begin{aligned}
        &\mathbb{E}^{\mathbb{P}^0}\big[\int_t^{T-t_0}\normy{\nabla(\mathcal{R}^\beta_{T-\theta}-\mathcal{R}^0_{T-\theta})(X_{T-\theta})}^2\,d\theta\big]\\
        &\quad=\mathbb{E}^{\mathbb{P}^0}\big[\log L^\beta_{T-t}(X_{T-t})-L^\beta_{T-t}(X_{T-t})+1+\int_t^{T-t_0}G_{T-\theta}(X_{T-\theta})\,d\theta\big],
    \end{aligned}\end{equation*}
    where the function $G_{T-t}(\cdot):\mathbb{R}^d\to\mathbb{R}$ is defined as
    \begin{equation*}
        G_{T-t}(x)\coloneqq(L^\beta_{T-t}(x)-1)\big(\beta\cdot\nabla\log\text{(}p^\beta_{T-t}L^\beta_{T-t})+\sum\limits_{i=1}^d\frac{\partial\beta^{(i)}}{\partial x_i}\big)(x),\quad\text{for all}\quad(t,x)\in[0,T-t_0]\times\mathbb{R}^d.
    \end{equation*}
    Introduce the constant
    \begin{equation*}
        C^{\prime\prime\prime\prime\prime}\coloneqq\sup\big\{\big|\beta\cdot\nabla\log\text{(}p^\beta_{T-t}L^\beta_{T-t})+\sum\limits_{i=1}^d\frac{\partial\beta^{(i)}}{\partial x_i}\big|:\,x\in\mathbb{R}^d,\,0\leq t\leq T-t_0\big\}<\infty.
    \end{equation*}
    Use (\ref{ineq 1, lem 5.2}), it is then immediate that
    \begin{equation*}
        \mathbb{E}^{\mathbb{P}^0}\big[\int_t^{T-t_0}\normy{\nabla(\mathcal{R}^\beta_{T-\theta}-\mathcal{R}^0_{T-\theta})(X_{T-\theta})}^2\,d\theta\big]\leq\mathbb{E}^{\mathbb{P}^0}\big[\int_t^{T-t_0}G_{T-\theta}(X_{T-\theta})\,d\theta\big]\leq C^{\prime\prime\prime\prime}C^{\prime\prime\prime\prime\prime}(T-t_0-t)^2.
    \end{equation*}
    Taking $C_2\coloneqq C^{\prime\prime\prime\prime}C^{\prime\prime\prime\prime\prime}$, (\ref{ineq 2, lem 5.2}) is verified.
\end{proof}
Having the Lemmas \ref{lem: 5.1} and \ref{lem: 5.6} as deviation control, we retrieve the known classical results on the relative entropy dissipation from the trajectorial approach in Theorems \ref{thm: Trajectorial rate of relative entropy dissipation, time-displacement, perturbed} and \ref{thm: Trajectorial rate of relative entropy dissipation, derivative}.
\begin{theorem}\label{thm: classical result, perturbed}
    Fix $t_0\geq0$ to be the time point when $\beta(\cdot)$ is initiated. We retrieve the known result on the forward-time dissipation of the relative entropy under perturbation,
    \begin{equation}\label{perturbed classical entropy decay, time-displacement}
        \mathbb{H}\big[P^\beta_t|Q\big]-\mathbb{H}\big[P^\beta_{t_0}|Q\big]=-\frac{1}{2}\mathbb{E}^{\mathbb{P}^\beta}\big[\int_{t_0}^t\normy{\nabla\mathcal{R}^\beta_{\theta}(X_{\theta})}^2 \,d\theta\big]+ \mathbb{E}^{\mathbb{P}^\beta}\big[\int_{t_0}^t\big(\sum\limits_{i=1}^d\frac{\partial\beta^{(i)}}{\partial x_i}-2\beta\cdot\nabla\psi\big)(X_{\theta})\,d\theta\big] 
    \end{equation}
    for all $t\geq t_0$. Furthermore, we also have the time-derivative, 
    \begin{equation}\label{perturbed classical entropy decay, derivative}
        \lim\limits_{t\searrow t_0} \frac{1}{t-t_0} \bigg(\mathbb{H}[P^\beta_t|Q]-\mathbb{H}[P^\beta_{t_0}|Q]\bigg)=-\frac{1}{2}\mathbb{I}[P^0_{t_0}|Q]-\mathbb{E}^{\mathbb{P}^0}\big[\big(\beta\cdot\nabla\mathcal{R}^0_{t_0}\big)(X_{t_0})\big].
    \end{equation}
\end{theorem}
\begin{proof}
    Following Theorem \ref{thm: Trajectorial rate of relative entropy dissipation, time-displacement, perturbed} and taking $\mathbb{P}^\beta$-expectation on (\ref{perturbed trajectorial entropy decay, time-displacement}), we get (\ref{perturbed classical entropy decay, time-displacement}). Since $\beta(\cdot)$ is smooth with compact support, the continuity of the sample paths of $(X_t)_{t\geq0}$ renders us,
    \begin{equation}\label{second term}
        \lim\limits_{t\searrow t_0} \frac{1}{t-t_0}\mathbb{E}^{\mathbb{P}^\beta}\big[\int_{t_0}^t\big(\sum\limits_{i=1}^d\frac{\partial\beta^{(i)}}{\partial x_i}-2\beta\cdot\nabla\psi\big)(X_{\theta})\,d\theta\big]=\mathbb{E}^{\mathbb{P}^\beta}\big[\big(\sum\limits_{i=1}^d\frac{\partial\beta^{(i)}}{\partial x_i}-2\beta\cdot\nabla\psi\big)(X_{t_0})\big].
    \end{equation}
    Notice that $X_{t_0}$ has the same distribution under $\mathbb{P}^\beta$ and $\mathbb{P}^0$. Henceforth, we can replace the $\mathbb{P}^\beta$-expectation in (\ref{second term}) with the $\mathbb{P}^0$-expectation. Moreover, integration by parts yields
    \begin{equation*}\begin{aligned}
        &\mathbb{E}^{\mathbb{P}^0}\big[\big(\sum\limits_{i=1}^d\frac{\partial\beta^{(i)}}{\partial x_i}-2\beta\cdot\nabla\psi\big)(X_{t_0})\big]=\int_{\mathbb{R}^d}\big(\sum\limits_{i=1}^d\frac{\partial\beta^{(i)}}{\partial x_i}-2\beta\cdot\nabla\psi\big)(x)p_{t_0}(x)\,dx\\
        &\quad\quad\quad=-\int_{\mathbb{R}^d}\beta\cdot\nabla(\log p_{t_0}+2\psi)(x)p_{t_0}(x)\,dx=-\mathbb{E}^{\mathbb{P}^0}\big[\big(\beta\cdot\nabla\mathcal{R}^0_{t_0}\big)(X_{t_0})\big].
    \end{aligned}\end{equation*}
    Applying (\ref{ineq 1, lem 5.2}) and (\ref{ineq 2, lem 5.2}),
    \begin{equation}\label{first term}
        \lim\limits_{t\searrow t_0} \frac{1}{t-t_0}\mathbb{E}^{\mathbb{P}^\beta}\big[\int_{t_0}^t\normy{\nabla\mathcal{R}^\beta_{\theta}(X_{\theta})}^2 \,d\theta\big]=\lim\limits_{t\searrow t_0} \frac{1}{t-t_0}\mathbb{E}^{\mathbb{P}^0}\big[\int_{t_0}^t\normy{\nabla\mathcal{R}^0_{\theta}(X_{\theta})}^2 \,d\theta\big]=\mathbb{I}\big[P^0_{T_0}|Q\big],
    \end{equation}
    where the last equality is due to Tschiderer \cite[Section 3.1]{Tschiderer}. Combining (\ref{second term}) and (\ref{first term}), the assertion (\ref{perturbed classical entropy decay, derivative}) is verified.
\end{proof}
When there is no perturbation, the time-derivative of the relative entropy dissipation conforms to the Fisher information modulo a multiplicative factor $1/2$. This unperturbed scenario will be further discussed in Section \ref{sec: wasserstein}, when we look at the steepest descent property of the relative entropy, from the Wasserstein space perspective.
\begin{corollary}\label{cor: classical result, unperturbed}
    Switching off the perturbation $\beta(\cdot):\mathbb{R}^d\to\mathbb{R}^d$, Theorem \ref{thm: classical result, perturbed} reduces to the classical result of the dissipation of relative entropy,
    \begin{equation}\label{unperturbed classical entropy decay, time-displacement}
        \mathbb{H}\big[P^0_t|Q\big]-\mathbb{H}\big[P^0_{t_0}|Q\big]=-\frac{1}{2}\mathbb{E}^{\mathbb{P}^0}\big[\int_{t_0}^t\normy{\nabla\mathcal{R}^0_{\theta}(X_{\theta})}^2 \,d\theta\big],\quad\text{for all}\quad t\geq t_0,
    \end{equation}
    as well as the classical limiting identity
    \begin{equation}\label{unperturbed classical entropy decay, derivative}
        \lim\limits_{t\searrow t_0}\frac{1}{t-t_0}\bigg(\mathbb{H}[P^0_t|Q]-\mathbb{H}[P^0_{t_0}|Q]\bigg)=-\frac{1}{2}\mathbb{I}[P^0_{t_0}|Q].
    \end{equation}
\end{corollary}
So far we have demonstrated that the classical consequences on the dissipation of relative entropy can be retrieved from the trajectorial formulation via Theorem \ref{thm: classical result, perturbed} and Corollary. However, it has not been answered why we choose the indirect, and probably less transparent, approach of time-reversal. In the remainder of this section, it will be shown that the backward-time approach is indeed more convenient than the forward-time approach.


\subsection{Defects in the forward-time approach}
We have mentioned in Section \ref{sec: intro} that the time-reversal principle is advantageous in its computational convenience. By comparing our derivation to the forward-time approach, we highlight where such computational convenience comes from.\par
We first compute the stochastic differential equations satisfied by the forward-time processes $(\ell^\beta_t(X_t))_{t\geq0}$ and $(\mathcal{R}^\beta_t(X_t))_{t\geq0}$. Notice that both processes are $(\mathcal{F}_t)_{t\geq0}$-adapted. Similar to the backward-time scenario, we use the Fokker-Planck equations (\ref{perturbed Fokker-Plank}) and (\ref{Fokker-Plank for Q}) to capture the differential structure of the likelihood ratio $\ell^\beta_t(x)=p^\beta_t(x)\exp(2\psi(x))$.
\begin{lemma}
    In the forward-time approach, the likelihood ratio process $(\ell^\beta_t(X_t))_{t\geq0}$ satisfies the stochastic differential equation,
    \begin{equation}\begin{aligned}\label{forward-time l}
        d\ell^\beta_t(X_t)&=\big(\sum\limits_{1\leq i\leq d}\frac{\partial\beta^{(i)}}{\partial x_i}-2\beta\cdot\nabla\psi\big)\ell^\beta_t(X_t)I_{\{t>t_0\}}\,dt\\
        &\quad+\big(\sum\limits_{1\leq i\leq d}\frac{\partial^2\ell^\beta_t}{\partial x_i^2}-2\nabla\ell^\beta_t\cdot\nabla\psi\big)(X_t)\,dt+\nabla\ell^\beta_t(X_t)\,dW^\beta_t,\quad\text{for all}\quad t\geq0.
    \end{aligned}\end{equation}
    And its logarithm, the forward-time relative entropy process $(\mathcal{R}^\beta_t(X_t))_{t\geq0}$ satisfies,
    \begin{equation}\begin{aligned}\label{forward-time R}
        d\mathcal{R}^\beta_t(X_t)&=\big(\sum\limits_{1\leq i\leq d}\frac{\partial\beta^{(i)}}{\partial x_i}-2\beta\cdot\nabla\psi\big)I_{\{t>t_0\}}\,dt+\big(\sum\limits_{1\leq i\leq d}\frac{1}{\ell^\beta_t}\frac{\partial^2\ell^\beta_t}{\partial x_i^2}-2\nabla\mathcal{R}^\beta_t\cdot\nabla\psi\big)(X_t)\,dt\\
        &\quad-\frac{1}{2}\normy{\nabla\mathcal{R}^\beta_t(X_t)}^d\,dt+\nabla\mathcal{R}^\beta_t(X_t)\,dW^\beta_t,\quad\text{for all}\quad t\geq0.
    \end{aligned}\end{equation}
\end{lemma}
\begin{proof}
    From the Fokker-Planck equations (\ref{perturbed Fokker-Plank}) and (\ref{Fokker-Plank for Q}), we can compute that
    \begin{equation}\label{forward-time PDE, l}
        \frac{\partial\ell^\beta_t}{\partial t}(x)=\frac{1}{2}\Delta\ell^\beta_t(x)+\nabla\ell^\beta_t\cdot(\beta I_{\{t>t_0\}}-\nabla\psi)(x)+\big(\sum\limits_{1\leq i\leq d}\frac{\partial\beta^{(i)}}{\partial x_i}-2\beta\cdot\nabla\psi\big)\ell^\beta_t(x)I_{\{t>t_0\}}.
    \end{equation}
    Via the Itô formula and that $\mathcal{R}^\beta_t(X_t)=\log\ell^\beta_t(X_t)$ for all $t\geq0$, the assertions (\ref{forward-time l}), (\ref{forward-time R}) are verified.
\end{proof}
Compared to the backward-time approach (\ref{semimartingale decomposition of R}), some extra terms show up in (\ref{forward-time R}). The analysis trajectorial behavior of the forward-time process $(\mathcal{R}^\beta_t(X_t))_{t\geq0}$ is therefore more involved. And consequently, the need of additional computation to deal with these extra terms makes the forward-time approach less transparent and eventually clouds the intuition of the phrase \textit{dissipation}.\par
We may still take $\mathbb{P}^\beta$-expectation, formally, to retrieve the classical identity of relative entropy dissipation through the forward-time approach. Indeed, after we verify the integrability of the additional term in (\ref{forward-time R}), performing the integration by parts shows that,
\begin{equation*}
    \mathbb{E}^{\mathbb{P}^\beta}\big[\sum\limits_{1\leq i\leq d}\frac{1}{\ell^\beta_t}\frac{\partial^2\ell^\beta_t}{\partial x_i^2}(X_t)-2(\nabla\mathcal{R}^\beta_t\cdot\nabla\psi)(X_t)\big]=0.
\end{equation*}
Henceforth, despite its computational complexity, this forward-time approach eventually leads to the same results, i.e.~Theorem \ref{thm: classical result, perturbed} and Corollary \ref{cor: classical result, unperturbed}. Nevertheless, we prefer to work on the interpretation to the trajectorial dynamics backward in time.



















\section{Connections to derivative in Wasserstein space}\label{sec: wasserstein}
The motivation of the Wasserstein space comes from a comparison between probability measures. In essence, the Wasserstein space is a suitably defined collection of probability measures endowed with a metric. An intuitive picture is to view each distribution as a unit amount of soil piled on ground. This metric quantifies the minimal cost of transporting one pile into the other, see Ambrosio/Gigli/Savaré \cite{Ambrosio/Gigli/Savaré}. By this analogy, the metric is known in Computer Science as the earth mover distance, see Levina/Bickel \cite{Levina/Bickel}.\par
The name, \textit{Wasserstein metric}, was coined by Dobrushin \cite{Dobrushin} after learning the work of Vaseršteĭn \cite{Vaseršteĭn} on Markov processes describing large systems of automata. Nevertheless, this metric has already been introduced by Kantorovich \cite{Kantorovich, Pass} in the context of Optimal Transport Theory. Wasserstein metric is a natural way to compare the laws of two random variables, where one is derived from the other by some perturbations, or undergoes time-evolution.\par
In our context of Itô-Langevin dynamics (\ref{perturbed Itô-Langevin dynamics}), $P^\beta_t$ and $P^\beta_{t_0}$ correspond to the marginal laws on $\mathbb{R}^d$ of $(X_t)_{t\geq t_0}$ at time $t$ and $t_0$. As previewed in Section \ref{sec: intro}, we present the time-derivative of $t\mapsto W_2(P^\beta_t,P^\beta_{t_0})$ at $t=t_0$. This limiting behavior reveals a correlation to the relative entropy dissipation via the quantity of Fisher information (\ref{Fisher information}). But first of all, let us proceed in an orderly way and start with the basic formulation of the quadratic Wasserstein space. 
\subsection{Basic structure of Wasserstein space}
Let $\mathscr{P}(\mathbb{R}^d)$ denote the set of all probability measures on the Borel sets of $\mathbb{R}^d$. In this expository article, the quadratic Wasserstein space is defined to be a metric space whose elements form a subset of $\mathscr{P}(\mathbb{R}^d)$. This metric structure quantifies the \textit{distance} between probability measures on $\mathbb{R}^d$. To be precise, the elements of the quadratic Wasserstein space $\mathscr{P}_2(\mathbb{R}^d)$ consist exactly of those elements in $\mathscr{P}(\mathbb{R}^d)$ with finite second moment, i.e.
\begin{equation}\label{def, quad, W space}
    \mathscr{P}_2(\mathbb{R}^d)\coloneqq\big\{P\in\mathscr{P}(\mathbb{R}^d):\,\int_{\mathbb{R}^d}\norm{x}^2\,dP(x)<\infty\big\},
\end{equation}
together with a metric function $W_2(\cdot,\cdot):\mathscr{P}_2(\mathbb{R}^d)\times\mathscr{P}_2(\mathbb{R}^d)\to\mathbb{R}_+$, which will be specified soon.\par
On the other hand, to simplify our exposition, occasionally we identify probability density functions on $\mathbb{R}^d$ with its associated Borel probability measures. Notice that if $p(\cdot):\mathbb{R}^d\to\mathbb{R}_+$ denotes a probability density function on $\mathbb{R}^d$, then its associated probability measure,
\begin{equation}
    \label{equivalence for density in W2}
    p(x)\,dx\in\mathscr{P}_2(\mathbb{R}^d)\qquad\text{if and only if}\qquad\int_{\mathbb{R}^d}x^2p(x)\,dx<\infty.
\end{equation}
In fact, if condition (\ref{equivalence for density in W2}) is satisfied, $p(\cdot):\mathbb{R}^d\to\mathbb{R}_+$ will be identified with an element in $\mathscr{P}_2(\mathbb{R}^d)$. Readers should stay alert to this convention. But in our expository article, this should leave no ambiguity.\par
Having specified the elements of the quadratic Wasserstein space $\mathscr{P}_2(\mathbb{R}^d)$ in (\ref{def, quad, W space}), we give a precise discription to the Wasserstein metric $W_2$. First of all, we adopt some notions and terminologies from the Optimal Transport Theory \cite{Agueh, Ambrosio/Gigli/Savaré}. Given $\mu,\nu\in\mathscr{P}(\mathbb{R}^d)$, let $\Gamma(\mu,\nu)$ denote the set of Kantorovich transport plans, i.e.~probability measures $\gamma$ on the Borel sets of $\mathbb{R}^d\times\mathbb{R}^d$ with marginals $\mu$ and $\nu$. Then, $\gamma$ satisfies $\pi^1_{\#}\gamma=\gamma\circ(\pi^1)^{-1}=\mu$ and $\pi^2_{\#}\gamma=\gamma\circ(\pi^2)^{-1}=\nu$, where $\pi^i:\mathbb{R}^d\times\mathbb{R}^d\to\mathbb{R}^d$, $i=1,2$ are the canonical projections. The Wasserstein metric $W_2$ is defined by,
\begin{equation}\label{quad W, metric}
    W_2(\mu,\nu)^2\coloneqq\inf\big\{\int_{\mathbb{R}^d\times\mathbb{R}^d}\norm{x-y}^2\,d\gamma(x,y):\,\gamma\in\Gamma(\mu,\nu)\big\},\quad\text{for all}\quad\mu,\nu\in\mathscr{P}_2(\mathbb{R}^d).
\end{equation}
It is verified that $W_2$ is indeed a metric on $\mathscr{P}_2(\mathbb{R}^d)$, see Sturm \cite{Sturm1, Sturm2}. In fact, the definition (\ref{quad W, metric}) of $W_2$ gives more regularity on the metric structure of $\mathscr{P}_2(\mathbb{R}^d)$. The quadratic Wasserstein space $\mathscr{P}_2(\mathbb{R}^d)$, equipped with the metric $W_2$, is separable and completely metrizable, i.e.~a Polish space, see Ambrosio/Gigli/Savaré \cite[Proposition 7.1.5]{Ambrosio/Gigli/Savaré}, \cite{Lott, Panaretos/Zemel}.\par
The Wasserstein metric $W_2$ is furthermore compatible with a Riemannian interpretation of the Wasserstein space \cite{Otto2, Otto}. Regarded formally as a Riemannian manifold consisting of Borel probability measures on $\mathbb{R}^d$, the characterization of $W_2$ suggests the tangent bundle $T\mathscr{P}_2(\mathbb{R}^d)\coloneqq\cup_\mu T_\mu\mathscr{P}_2(\mathbb{R}^d)$ to $\mathscr{P}_2(\mathbb{R}^d)$, where
\begin{equation}\label{quad W, tangent space}
    T_\mu\mathscr{P}_2(\mathbb{R}^d)\coloneqq\overline{\big\{\nabla\varphi:\,\varphi\in\mathcal{C}^\infty_c(\mathbb{R}^d;\mathbb{R})\big\}}^{L^2(\mu)},\quad\text{for all}\quad\mu\in\mathscr{P}_2(\mathbb{R}^d).
\end{equation}
Naturally, (\ref{quad W, tangent space}) hints to a differential structure to the Wasserstein metric framework of $\mathscr{P}_2(\mathbb{R}^d)$.\par
In light of the Riemannian structure of $\mathscr{P}_2(\mathbb{R}^d)$, we can talk of the \textit{constant speed geodesic}. Indeed, this concept is studied in Differential Geometry, where any two points on a smooth manifold are connected by a unique length-minimized curve, called the \textit{geodesic} \cite{Hirsch, Lee, Tu}. In the Wasserstein space theory, (\ref{quad W, tangent space}) provides a tangent bundle to $\mathscr{P}_2(\mathbb{R}^d)$ and hence defines its manifold structure. Then, given two arbitrary probability measures in $\mathscr{P}_2(\mathbb{R}^d)$, the scheme to transport one probability measure to the other with minimal effort, i.e.~cumulative tangential distance, corresponds exactly to a geodesic on $\mathscr{P}_2(\mathbb{R}^d)$. This geodesic can be also written as a parametrized family of probability measures in $\mathscr{P}_2(\mathbb{R}^d)$.\par
Let $I\coloneqq[a,b]$ be the parameter interval. Fix $\mu_a,\mu_b\in\mathscr{P}_2(\mathbb{R}^d)$. If we can find a transport map $\mathcal{T}^G\coloneqq\nabla G:\mathbb{R}^d\to\mathbb{R}^d$ such that $\mu_b=(\mathcal{T}^G)_\#\mu_a=\mu_a\circ(\mathcal{T}^G)^{-1}$ and $G$ is convex on $\mathbb{R}^d$, then the parametrized family $(\mu_t)_{t\in I}$ in $\mathscr{P}_2(\mathbb{R}^d)$ defined by,
\begin{equation*}
    \mu_t\coloneqq(\mathcal{T}^G_t)_\#\mu_a,\qquad\text{where}\quad\mathcal{T}^G_t\coloneqq \frac{b-t}{b-a} Id_{\mathbb{R}^d}+\frac{t-a}{b-a}\nabla G,\quad\text{for all}\quad t\in[a,b],
\end{equation*}
is a curve in $\mathscr{P}_2(\mathbb{R}^d)$ connecting $\mu_a$ and $\mu_b$ which, for all $a\leq u\leq v\leq b$, satisfies,
\begin{equation}\label{optimal transport}
    W_2(\mu_u,\mu_v)=\frac{v-u}{b-a}\sqrt{\int_{\mathbb{R}^d}\normy{x-\nabla G(x)}^2\,d\mu_a(x)}=\frac{v-u}{b-a}\normy{x-\nabla G(x)}_{L^2(\mu_a)}.
\end{equation}
And this parametrized curve $(\mu_t)_{t\in I}$ in $\mathscr{P}_2(\mathbb{R}^d)$ is the constant speed geodesic from $\mu_a$ to $\mu_b$. The result (\ref{optimal transport}) is the \textit{Brenier theorem} for the Wasserstein spaces. Readers are referred to Brenier \cite[Section 3]{Brenier} and Villani \cite[Theorem 2.12]{Villani} for more details. In practice, we often construct a transport map $\mathcal{T}^G$ from a convex function $G(\cdot):\mathbb{R}^d\to\mathbb{R}$. Then, this $(\mu_t)_{t\in I}$ is automatically a constant speed geodesic in $\mathscr{P}_2(\mathbb{R}^d)$.\par
In view of the Wasserstein space theory, the family $(P^\beta_t)_{t\geq0}$ can be equivalently seen as a parametrized curve in the manifold $\mathscr{P}_2(\mathbb{R}^d)$, in light of Lemma \ref{lem: strong solution, perturbed Itô-Langevin}. And we further impose some regularity conditions on the potential $\psi(\cdot)$ which drives the Itô-Langevin dynamics (\ref{perturbed Itô-Langevin dynamics}) in this context. At each $t\geq0$, we assume that $\psi(\cdot)$ is chosen such that there exists a sequence of functions $(\varphi^{0,(m)}_t(\cdot))_{m\in\mathbb{N}}$ of class $\mathcal{C}^\infty(\mathbb{R}^d;\mathbb{R})$ with compact support, whose gradients $(\nabla\varphi^{0,(m)}_t(\cdot))_{m\in\mathbb{N}}$ admit the mean square convergence
    \[
        \nabla\varphi^{0,(m)}_t(\cdot)\xrightarrow{\,L^2(\mathbb{R}^d,P^0_t)\,}V^0_t(\cdot)\coloneqq\nabla\varphi^0_t(\cdot)\quad\text{as}\quad m\to\infty,
    \]
    where the time-dependent velocity field $V^0_t(\cdot)$ is of gradient type with $\varphi^0_t(\cdot)\coloneqq-\psi(\cdot)-\tfrac{1}{2}\log p^0_t(\cdot)$. Here, $P^0_t$ corresponds to the unperturbed marginal distribution of $X_t$ in (\ref{perturbed Itô-Langevin dynamics}) at $t\geq0$. In particular, at $t_0$,
    \begin{equation*}
        V^0_{t_0}(\cdot)\in T_{P^0_{t_0}}\mathscr{P}_2(\mathbb{R}^d)=\overline{\big\{\nabla\varphi(\cdot):\,\varphi\in\mathcal{C}_c^\infty(\mathbb{R}^d,\mathbb{R})\big\}}^{L^2(P^0_{t_0})}.
    \end{equation*}
    When there is perturbation, denote by $V^\beta_t(\cdot)\coloneqq\nabla\varphi^\beta_t(\cdot)$ with $\varphi^\beta_t(\cdot)\coloneqq-(\psi+BI_{\{t>t_0\}})(\cdot)-\tfrac{1}{2}\log p^\beta_t(\cdot)$. Since $\beta(\cdot)$ is of gradient type, i.e.~$\beta=\nabla B$, where $B(\cdot)$ is smooth and compactly supported, it is clear that
    \begin{equation}\label{tangent space condition}
        V^\beta_{t_0}(\cdot)\in\text{T}_{P^0_{t_0}}\mathscr{P}_2(\mathbb{R}^d)=\overline{\big\{(\nabla\varphi+\beta)(\cdot):\,\varphi\in\mathcal{C}_c^\infty(\mathbb{R}^d,\mathbb{R})\big\}}^{L^2(P^0_{t_0})}.
    \end{equation}
In practice, the expression (\ref{tangent space condition}) ensures we could find a sequence of compactly supported smooth vector fields to approximate $V^\beta_{t_0}(\cdot)$, all of which are of gradient type.
\subsection{Local behavior of Wasserstein metric} Having introduced the Wasserstein spaces, let us turn to the limiting behavior of $(t-t_0)^{-1}W_2(P^\beta_{t},P^\beta_{t_0})$ as $t\searrow t_0$. This limiting identity is stated in Theorem \ref{w-space displacement}. When the perturbation vanishes, this limiting identity reduces to an expression of the Fisher information quantity (\ref{Fisher information}) and is therefore correlated to the time-derivative (\ref{perturbed classical entropy decay, derivative}) of the relative entropy.\par
For the clarity of our exposition, we first claim that a family of random variables are $\mathbb{P}^\beta$-uniformly integrable. An inspection of (\ref{family of r.v.}) tells that these random variables are all functionals of the velocity field $(V^\beta_t(\cdot))_{t\geq0}$. And their uniform integrability is important to the Lemma \ref{lem, 6.2}. Readers are encouraged to go through the proof of Lemma \ref{lem: 6.1}, but it is also fine to skip its proof in first reading of this section.
\begin{lemma}\label{lem: 6.1}
    The family of random variables,
    \begin{equation}\label{family of r.v.}
        \bigg(\normy{\frac{1}{t-t_0}\int_{t_0}^tV^\beta_\theta(X_\theta)\,d\theta-V^\beta_{t_0}(X_{t_0})}^2\bigg)_{t\geq t_0},
    \end{equation}
    is $\mathbb{P}^\beta$-uniformly integrable.
\end{lemma}
\begin{proof}
    Notice that for each $t\geq0$, the velocity vector $V^\beta_t(\cdot)$ is of class $\mathcal{C}^\infty(\mathbb{R}^d;\mathbb{R}^d)$ with compact support. Hence we know $\normx{V^\beta_{t}(X_{t})}^2\in L^1(\mathbb{P}^\beta)$ for all $t\geq0$, and by the Jensen inequality we have
    \begin{equation*}
        \normy{\frac{1}{t-t_0}\int_{t_0}^tV^\beta_\theta(X_\theta)\,d\theta}^2\leq\frac{1}{t-t_0}\int_{t_0}^t\normy{V^\beta_\theta(X_\theta)}^2\,d\theta.
    \end{equation*}
    It then suffices to prove the uniform integrability of the family
    \begin{equation*}
        \bigg(\frac{1}{t-t_0}\int_{t_0}^t\normy{V^\beta_\theta(X_\theta)}^2\,d\theta\bigg)_{t\geq t_0}.
    \end{equation*}
    Invoking the definition of the velocity field $V^\beta_t(\cdot)$ and that $\beta(\cdot)$ is smooth with compact support,
    \begin{equation*}
        \bigg(\frac{1}{t-t_0}\int_{t_0}^t\normy{V^\beta_\theta(X_\theta)}^2\,d\theta\bigg)_{t\geq t_0}\quad\text{is U.I. if and only if}\quad\bigg(\frac{1}{t-t_0}\int_{t_0}^t\normy{\nabla\mathcal{R}^\beta_\theta(X_\theta)}^2\,d\theta\bigg)_{t\geq t_0}\quad\text{is U.I.,}
    \end{equation*}
    where U.I. abbreviates the phrase \textit{uniformly integrable}. By continuity of the sample paths of $(X_t)_{t\geq0}$,
    \begin{equation*}
        \frac{1}{t-t_0}\int_{t_0}^t\normy{\nabla\mathcal{R}^\beta_\theta(X_\theta)}^2\,d\theta\to\normy{\nabla\mathcal{R}^0_{t_0}(X_{t_0})}^2\;\;\text{as}\quad t\to t_0,\quad\mathbb{P}^\beta\text{-a.s.}
    \end{equation*}
    Since $L^1(\mathbb{P}^\beta)$ convergence implies $\mathbb{P}^\beta$-uniform integrability, it suffices to check the convergence of their $\mathbb{P}^\beta$-expectation by the Scheffé lemma. In fact, using (\ref{first term}), we ascertain this claim. 
\end{proof}
The $\mathbb{P}^\beta$-uniform integrability of the random variables (\ref{family of r.v.}) is necessary to the proof of Lemma \ref{lem, 6.2}, where we transfer the $\mathbb{P}^0$-a.s. convergence of a sequence of random variables to their corresponding $L^2(\mathbb{P}^0)$-convergence in (\ref{eqn: 6.8}).
\begin{lemma}\label{lem, 6.2}
    The velocity field $(V^\beta_t(\cdot))_{t\geq t_0}$ induces a curved flow $(\mathcal{L}^\beta_t)_{t\geq t_0}$, characterized by
    \begin{equation*}
        \mathcal{L}^\beta_{t_0}=Id_{\mathbb{R}^d}\qquad\text{and}\qquad\frac{d}{dt}\mathcal{L}^\beta_t=V^\beta_t(\mathcal{L}^\beta_t),\quad\text{for all}\quad t\geq t_0.
    \end{equation*}
    Then, for all $t\geq t_0$, $(\mathcal{L}^\beta_t)_\#P^\beta_{t_0}=P^\beta_t$, i.e.~the map $\mathcal{L}^\beta_t:\mathbb{R}^d\to\mathbb{R}^d$ transports the probability measure $P^\beta_{t_0}=P^0_{t_0}$ to the probability measure $P^\beta_t$. Moreover,
    \begin{equation}\label{eqn: 6.8}
        \lim\limits_{t\searrow t_0}\frac{1}{t-t_0}\mathbb{E}^{\mathbb{P}^0}\big[\normy{\mathcal{L}^\beta_t(X_{t_0})-X_{t_0}-(t-t_0)V^\beta_{t_0}(X_{t_0})}^2\big]^{1/2}=0.
    \end{equation}
\end{lemma}
\begin{proof}
    First note that 
    \begin{equation*}
        \mathcal{L}^\beta_t(x)=x+\int_{t_0}^tV^\beta_\theta(\mathcal{L}^\beta_\theta(x))\,d\theta,\quad\text{for all}\quad(t,x)\in[t_0,\infty)\times\mathbb{R}^d.
    \end{equation*}
    On this account, for all $t\geq t_0$,
    \begin{equation*}
        \mathbb{E}^{\mathbb{P}^0}\big[\normy{\mathcal{L}^\beta_t(X_{t_0})-X_{t_0}-(t-t_0)V^\beta_{t_0}(X_{t_0})}^2\big]=\mathbb{E}^{\mathbb{P}^0}\big[\normy{\int_{t_0}^tV^\beta_\theta(\mathcal{L}^\beta_\theta(x))\,d\theta-(t-t_0)V^\beta_{t_0}(X_{t_0})}^2\big].
    \end{equation*}
    In light of Lemma \ref{lem: 5.1}, $\mathbb{P}^\beta$ and $\mathbb{P}^0$ are mutually absolutely continuous with uniformly bounded density process. Hence, to verify (\ref{eqn: 6.8}), it suffices to show the limiting assertion,
    \begin{equation*}
        \lim\limits_{t\searrow t_0}\frac{1}{t-t_0}\mathbb{E}^{\mathbb{P}^\beta}\big[\normy{\int_{t_0}^tV^\beta_\theta(X_\theta)\,d\theta-(t-t_0)V^\beta_{t_0}(X_{t_0})}^2\big]^{1/2}=0.
    \end{equation*}
    By the continuity of the sample paths of $(X_t)_{t\geq0}$,
    \begin{equation*}
        \normy{\frac{1}{t-t_0}\int_{t_0}^tV^\beta_\theta(X_\theta)\,d\theta-V^\beta_{t_0}(X_{t_0})}^2\to0\quad\text{as}\quad t\to t_0,\quad\mathbb{P}^\beta\text{-a.s.}
    \end{equation*}
    By the $\mathbb{P}^\beta$-uniform integrability in Lemma \ref{lem: 6.1}, the convergence still holds after taking $\mathbb{P}^\beta$-expectation.
\end{proof}
The limiting assertion (\ref{eqn: 6.8}) to the non-optimal transport plan $(\mathcal{L}^\beta_t)_{t\geq t_0}$ will be used in Theorem \ref{w-space displacement}, where we decompose the transport from $P^\beta_{t_0}$ to $P^\beta_t$ into a composition of a sequence of optimal transport plans $(\mathcal{J}^{\beta,(m)}_t)_{t\geq t_0, m\in\mathbb{N}}$ and the non-optimal transport $(\mathcal{L}^\beta_t)_{t\geq t_0}$.
\begin{theorem}\label{w-space displacement}
    We have the local limiting behavior of the quadratic Wasserstein metric,
    \begin{equation}\label{eqn, w-space distance}
        \lim\limits_{t\searrow t_0}\frac{1}{t-t_0}W_2(P^\beta_t,P^\beta_{t_0})=\frac{1}{2}\normy{\nabla\mathcal{R}^0_{t_0}(X_{t_0})+2\beta(X_{t_0})}_{L^2(\mathbb{P}^0)}.
    \end{equation}
\end{theorem}
\begin{proof}
    According to (\ref{tangent space condition}), there exists a sequence of compactly supported functions $(\varphi^{\beta,(m)}_{t_0}(\cdot))_{m\in\mathbb{N}}$ of class $\mathcal{C}^\infty(\mathbb{R}^d;\mathbb{R})$, such that
    \begin{equation}\label{L2 limit of velocity vector}
        \lim\limits_{m\to\infty}\mathbb{E}^{\mathbb{P}^0}\big[\normy{V^\beta_{t_0}(X_{t_0})-\nabla\varphi^{\beta,(m)}_{t_0}(X_{t_0})}^2\big]=0.
    \end{equation}
    We call the gradients $(\nabla\varphi^{\beta,(m)}_{t_0}(\cdot))_{m\in\mathbb{N}}$ the \textit{localized gradient fields}, which have compact support and approximate the velocity field $V^\beta_{t_0}(\cdot)$ in $L^2(\mathbb{P}^0)$. These localized gradient fields induce a sequence of \textit{localized linear transports} $(\mathcal{J}^{\beta,(m)}_t)_{t\geq t_0,m\in\mathbb{N}}$, defined by
    \begin{equation*}
        \mathcal{J}^{\beta,(m)}_t(x)\coloneqq x+(t-t_0)\nabla\varphi^{\beta,(m)}_{t_0}(x)\quad\text{for all}\quad x\in\mathbb{R}^d,\quad t\geq t_0,\quad\text{and}\quad m\in\mathbb{N}.
    \end{equation*}
    Denote by $P^{\beta,(m)}_{\mathcal{J}_t}$ the transport image of $P_{t_0}$ under $\mathcal{J}^{\beta,(m)}_t$, i.e.~$P^{\beta,(m)}_{\mathcal{J}_t}=(\mathcal{J}^{\beta,(m)}_t)_\# P_{t_0}$ for all $t\geq t_0$ and $m\in\mathbb{N}$. We claim that
    \begin{equation}\label{optimal local transport}
        \lim\limits_{t\searrow t_0}\frac{1}{t-t_0}W_2(P^{\beta,(m)}_{\mathcal{J}_t},P_{t_0})=\normy{\nabla\varphi^{\beta,(m)}_{t_0}(X_{t_0})}_{L^2(\mathbb{P}^0)}.
    \end{equation}
    In order to deduce (\ref{optimal local transport}), we have to show that $\mathcal{J}^{\beta,(m)}_t(\cdot):\mathbb{R}^d\to\mathbb{R}^d$ is the gradient of a convex function, for all $t\geq t_0$ sufficiently close to $t_0$. From its definition,
    \begin{equation*}
        \mathcal{J}^{\beta,(m)}_t(x)=\nabla\big(\frac{1}{2}\norm{x}^2+(t-t_0)\varphi^{\beta,(m)}_{t_0}(x)\big)\quad\text{for all}\quad x\in\mathbb{R}^d.
    \end{equation*}
    Hence, it suffices to show that $\tfrac{1}{2}\norm{\cdot}^2+(t-t_0)\varphi^{\beta,(m)}_{t_0}(\cdot)$ is convex for all $m\in\mathbb{N}$, when $t\geq t_0$ is close enough to $t_0$. Its Hessian matrix is given by,
    \begin{equation}\label{hessioan of matrix phi}
        Id_{\mathbb{R}^d}+(t-t_0)\text{Hess}(\varphi^{\beta,(m)}_{t_0})(x),\quad\text{for all}\quad x\in\mathbb{R}^d.
    \end{equation}
    Since $\varphi^{\beta,(m)}_{t_0}(\cdot)$ is smooth with compact support, there exists $\epsilon_m>0$ such that (\ref{hessioan of matrix phi}) is positive definite for all $t_0\leq t\leq t_0+\epsilon_m$, uniformly in $x\in\mathbb{R}^d$. Hence, for each $m\in\mathbb{N}$, $\mathcal{J}^{\beta,(m)}_t(\cdot)$ is indeed the gradient of a convex function when $t_0\leq t\leq t_0+\epsilon_m$. And (\ref{optimal local transport}) follows from the Brenier theorem, \cite[Section 3]{Brenier}, \cite[Theorem 2.12]{Villani}. Invoking (\ref{L2 limit of velocity vector}),
    \begin{equation}\label{w limit 1}
        \lim\limits_{m\to\infty}\lim\limits_{t\searrow t_0}\frac{1}{t-t_0}W_2(P^{\beta,(m)}_{\mathcal{J}_t},P_{t_0})=\normy{V^\beta_{t_0}(X_{t_0})}_{L^2(\mathbb{P}^0)}=\frac{1}{2}\normy{\nabla\mathcal{R}^0_{t_0}(X_{t_0})+2\beta(X_{t_0})}_{L^2(\mathbb{P}^0)}.
    \end{equation}
    Our next step is to show that
    \begin{equation}\label{w limit 2}
        \lim\limits_{m\to\infty}\lim\limits_{t\searrow t_0}\frac{1}{t-t_0}W_2(P^\beta_t,P^{\beta,(m)}_{\mathcal{J}_t})=0.
    \end{equation}
    To achieve this, we construct a transport plan from $P^{\beta,(m)}_{\mathcal{J}_t}$ to $P^\beta_t$. In Lemma \ref{lem, 6.2}, we have the non-optimal transport $\mathcal{L}^\beta_t$ with $(\mathcal{L}^\beta_t)_\#P_{t_0}=P^\beta_t$. And we have the localized linear transport $\mathcal{J}^{\beta,(m)}_t$ with $(\mathcal{J}^{\beta,(m)}_t)_\# P_{t_0}=P^{\beta,(m)}_{\mathcal{J}_t}$. To this end, let $\mathcal{H}^{\beta,(m)}_t\coloneqq\mathcal{L}^\beta_t\circ(\mathcal{J}^{\beta,(m)}_t)^{-1}$, whence $(\mathcal{H}^{\beta,(m)}_t)_\#P^{\beta,(m)}_{\mathcal{J}_t}=P^\beta_t$ for all $t\geq t_0$. Let $\mathbb{P}^{\beta,(m)}_{\mathcal{J}}$ denote a probability measure on the path space $\mathcal{C}$ under which the canonical coordinate process $(X_t)_{t\geq0}$ has the marginal distribution $P^{\beta,(m)}_{\mathcal{J}_t}$ at each $t\geq t_0$ and such that the marginals of $\mathbb{P}^{\beta,(m)}_{\mathcal{J}}$ agrees with $\mathbb{P}$ at time $t$ when $0\leq t\leq t_0$. Then,
    \begin{equation*}
        \mathbb{E}^{\mathbb{P}^{\beta,(m)}_{\mathcal{J}}}\big[\normy{\mathcal{H}^{\beta,(m)}_t(X_t)-X_t}^2\big]=\mathbb{E}^{\mathbb{P}^0}\big[\normy{\mathcal{L}^\beta_t(X_{t_0})-\mathcal{J}^{\beta,(m)}_t(X_{t_0})}^2\big],\quad\text{for all}\quad t\geq t_0.
    \end{equation*}
    Notice that,
    \begin{equation*}
        \frac{1}{2(t-t_0)^2}\normy{\mathcal{L}^\beta_t(x)-\mathcal{J}^{\beta,(m)}_t(x)}^2\leq\normy{V^\beta_{t_0}(x)-\nabla\varphi^{\beta,(m)}_{t_0}(x)}^2+\normy{\frac{1}{t-t_0}\int_{t_0}^tV^\beta_\theta(\mathcal{L}^\beta_\theta(x))\,d\theta-V^\beta_{t_0}(x)}^2.
    \end{equation*}
    Using (\ref{L2 limit of velocity vector}) and Lemma \ref{lem, 6.2}, we can conclude that
    \begin{equation*}
        \lim\limits_{m\to\infty}\lim\limits_{t\searrow t_0}\frac{1}{t-t_0}W_2(P^\beta_t,P^{\beta,(m)}_{\mathcal{J}_t})=\lim\limits_{m\to\infty}\lim\limits_{t\searrow t_0}\frac{1}{t-t_0}\mathbb{E}^{\mathbb{P}^{\beta,(m)}_{\mathcal{J}}}\big[\normy{\mathcal{H}^{\beta,(m)}_t(X_t)-X_t}^2\big]^{1/2}=0
    \end{equation*}
    which verifies (\ref{w limit 2}). Since
    \begin{equation*}
        \lim\limits_{m\to\infty}\lim\limits_{t\searrow t_0}\frac{1}{t-t_0}W_2(P^{\beta,(m)}_{\mathcal{J}_t},P_{t_0})\leq\lim\limits_{m\to\infty}\lim\limits_{t\searrow t_0}\frac{1}{t-t_0}W_2(P^{\beta,(m)}_{\mathcal{J}_t},P^\beta_t)+\liminf\limits_{t\searrow t_0}\frac{1}{t-t_0}W_2(P^\beta_t,P_{t_0})
    \end{equation*}
    as well as
    \begin{equation*}
        \limsup\limits_{t\searrow t_0}\frac{1}{t-t_0}W_2(P^\beta_t,P_{t_0})\leq\lim\limits_{m\to\infty}\lim\limits_{t\searrow t_0}\frac{1}{t-t_0}W_2(P^\beta_t,P^{\beta,(m)}_{\mathcal{J}_t})+\lim\limits_{m\to\infty}\lim\limits_{t\searrow t_0}\frac{1}{t-t_0}W_2(P^{\beta,(m)}_{\mathcal{J}_t},P_{t_0}),
    \end{equation*}
    using (\ref{w limit 1}) and (\ref{w limit 2}) we can conclude (\ref{eqn, w-space distance}). And the assertion is verified.
\end{proof}
Theorem \ref{w-space displacement} reveals the time-derivative of the Wasserstein metric from $(P^\beta_t)_{t\geq t_0}$ to $P^\beta_{t_0}$. This limiting identity (\ref{eqn, w-space distance}) includes both the gradient of the relative entropy process and perturbation terms. In fact, if we collapse the perturbation, the results become more transparent.
\begin{corollary}\label{w-space displacement, unperturbed}
    Switching off the perturbation $\beta(\cdot):\mathbb{R}^d\to\mathbb{R}^d$, Theorem \ref{w-space displacement} reduces to the time-derivative of the unperturbed Wasserstein metric from $(P^0_t)_{t\geq t_0}$ to $P^0_{t_0}$,
    \begin{equation}\label{eqn, 6.15}
        \lim\limits_{t\searrow t_0}\frac{1}{t-t_0}W_2(P^0_t,P^0_{t_0})=\frac{1}{2}\norm{\nabla\mathcal{R}^0_{t_0}(X_{t_0})}_{L^2(\mathbb{P}^0)}.
    \end{equation}
\end{corollary}
Without perturbation, the time-derivative of the Wasserstein metric (\ref{eqn, 6.15}) is equal to the square root of Fisher information. Through this limiting identity, (\ref{eqn, 6.15}) is therefore correlated to the time-derivative of the relative entropy (\ref{unperturbed classical entropy decay, time-displacement}). Additionally, this insight reveals the steepest descent property of the dissipation of relative entropy, concerning the scenario of the unperturbed dynamics.
\subsection{Steepest descent property of relative entropy $\mathbb{H}$}
The philosophy of steepest descent is to locate a parametrized curve from an abstract manifold, such that the varying rate of some indexed quantities is extremized. This idea was adopted by Debye \cite{Debye} who used Bessel functions \cite{Temme} to numerically approximate an integral. In the work of Lagrange \cite{Lagrange1, Lagrange2}, Landau/Lifshitz \cite{Landau/Lifshitz}, and Feynman \cite{Feynman}, the Lagrangian formalism of mechanics was progressively designed to interpret the variational principles and the trajectory of classical particles.\par
Over this expository article, the phrase \textit{steepest descent} has appeared without an explanation. What it refers to is not completely in align with the literature listed above. Nonetheless, its precise interpretation will be clarified at this point. And this steepest descent property, corresponding to the unperturbed scenario of (\ref{perturbed Itô-Langevin dynamics}), will also answer the question why we are interested in introducing the smooth perturbation $\beta(\cdot)$ into our Itô-Langevin dynamics, and why the unperturbed case is remarkable.\par
Combining the results from Theorems \ref{thm: classical result, perturbed} and \ref{w-space displacement}, we observe that the time-derivatives of relative entropy and Wasserstein metric, evaluated at $t_0\geq0$, are correlated via an expression of Fisher information (\ref{Fisher information}) as well as some perturbation terms, i.e.
\begin{equation}
    \label{derivetive: composed, perturbed}
    \lim\limits_{t\searrow t_0}\frac{\,\mathbb{H}[P^\beta_t|Q]-\mathbb{H}[P^\beta_{t_0}|Q]\,}{W_2(P^\beta_t,P^\beta_{t_0})}=-\mathbb{E}^{\mathbb{P}^0}\bigg[\nabla\mathcal{R}^0_{t_0}(X_{t_0})\cdot\frac{\nabla\mathcal{R}^0_{t_0}(X_{t_0})+2\beta(X_{t_0})}{\norm{\nabla\mathcal{R}^0_{t_0}(X_{t_0})+2\beta(X_{t_0})}_{L^2(\mathbb{P}^0)}}\bigg].
\end{equation}
When the perturbation vanishes, the RHS of (\ref{derivetive: composed, perturbed}) reduces to the square root of Fisher information, i.e.
\begin{equation}
    \label{derivetive: composed, unperturbed}
    \lim\limits_{t\searrow t_0}\frac{\,\mathbb{H}[P^0_t|Q]-\mathbb{H}[P^0_{t_0}|Q]\,}{W_2(P^0_t,P^0_{t_0})}=-\mathbb{E}^{\mathbb{P}^0}\bigg[\nabla\mathcal{R}^0_{t_0}(X_{t_0})\cdot\frac{\nabla\mathcal{R}^0_{t_0}(X_{t_0})}{\norm{\nabla\mathcal{R}^0_{t_0}(X_{t_0})}_{L^2(\mathbb{P}^0)}}\bigg].
\end{equation}
Comparing (\ref{derivetive: composed, perturbed}) and (\ref{derivetive: composed, unperturbed}) and in light of the Cauchy-Schwarz inequality, we observe that their difference,
\begin{equation*}
    \lim\limits_{t\searrow t_0}\frac{\,\mathbb{H}[P^\beta_t|Q]-\mathbb{H}[P^\beta_{t_0}|Q]\,}{W_2(P^\beta_t,P^\beta_{t_0})}-\lim\limits_{t\searrow t_0}\frac{\,\mathbb{H}[P^0_t|Q]-\mathbb{H}[P^0_{t_0}|Q]\,}{W_2(P^0_t,P^0_{t_0})},
\end{equation*}
is always nonnegative, and strictly positive when $\beta(\cdot)$ is not parallel to $\nabla\mathcal{R}^0_{t_0}$. If we otherwise view $\mathbb{H}[P^\beta_t|Q]$, $t\geq0$ as a flow on the curve $(P^\beta_t)_{t\geq0}\subseteq\mathscr{P}_2(\mathbb{R}^d)$, then its slope reaches infimum in the absence of perturbation. Henceforth, the relative entropy is unlikely to increase and most likely to decrease when the perturbation vanishes. This extremal phenomenon is therefore referred as the steepest descent property.


\subsection{Dissipative velocity of relative entropy $\mathbb{H}$}
Applying additional non-degeneracy conditions on the second-order derivatives of the potential $\psi(\cdot)$, we could extract more information from the unperturbed Itô-Langevin stochastic dynamics (\ref{perturbed Itô-Langevin dynamics}). Namely, we obtain the Bakry-Émery \cite{Bakry/Émery} exponential decay rate of $\mathbb{H}[P^0_t|Q]$. For an invitation to the relevant topics in the Bakry-Émery theory, which derives also the exponential decay of $\mathbb{I}[P^0_t|Q]$ defined in (\ref{Fisher information}), readers are encouraged to the references Bakry-Émery \cite{Bakry/Émery}, Bakry/Gentil/Ledoux \cite{Bakry/Gentil/Ledoux}, and Gentil \cite{Gentil}.\par
In our exposition, the derivation relies on the analysis of the geodesics in $\mathscr{P}_2(\mathbb{R}^d)$, viewed as a manifold. First, let $\mu_a,\mu_b$, $a<b$, be two elements in $\mathscr{P}_2(\mathbb{R}^d)$, both absolutely continuous with respect to the reference measure $Q$ introduced in Section \ref{sec: stochastic dynamics}. Let $\mathcal{T}$ denote the optimal transport form $\mu_a$ to $\mu_b$, i.e.~$\mu_b=(\mathcal{T})_\#\mu_a$. Then, the interpolation family $(\mathcal{T}_t)_{a\leq t\leq b}$ of transport plans induced by $\mathcal{T}$ such that
\begin{equation*}
    \mathcal{T}_t\coloneqq\frac{b-t}{b-a} Id_{\mathbb{R}^d}+\frac{t-a}{b-a}\mathcal{T},\quad\text{for all}\quad t\in[a,b],
\end{equation*}
generates a law $\mu$ on the Borel sets of $\mathcal{C}([a,b];\mathbb{R}^d)$. For each $t\in[a,b]$, let $\mu_t$ denote the marginal distribution of $\mu$ on $\mathbb{R}^d$ at $t$. Then, the parametrized family $(\mu_t)_{a\leq t\leq b}$ is a curve in $\mathscr{P}_2(\mathbb{R}^d)$ and satisfies $\mu_t\coloneqq(\mathcal{T}_t)_\#\mu_a$ for all $t\in[a,b]$. Hence, $(\mu_t)_{a\leq t\leq b}$ is a constant speed geodesic.\par
In Section \ref{sec: stochastic dynamics}, we have defined the relative entropy $\mathbb{H}[P^\beta_t|Q]$ and its associated process $\mathcal{R}^{\mathbb{P}^\beta}_t$, abbreviated as $\mathcal{R}^\beta_t$, for each $t\geq0$. We can similarly define the relative entropy $\mathbb{H}[\mu_t|Q]$ as well as its associated process $\mathcal{R}^\mu_t(X_t)\coloneqq\log d\mu_t/dQ$ with $a\leq t\leq b$, where $(X_t)_{a\leq t\leq b}$ is the canonical coordinate process in $\mathcal{C}([a,b];\mathbb{R}^d)$ so that $X_t\sim\mu_t$ for all $a\leq t\leq b$. The dissipation of $\mathbb{H}[\mu_t|Q]$, calculated against $Q$ and along the geodesic $(\mu_t)_{a\leq t\leq b}$, will be our first step to understand the exponential decay rate of $\mathbb{H}[P^0_t|Q]$ with $t\geq t_0$.
\begin{lemma}\label{lem: 6.6}
    Fix $\mu_a,\mu_b\in\mathscr{P}_2(\mathbb{R}^d)$, both absolutely continuous with respect to $Q$. We have the time-derivative of the relative entropy along the constant speed geodesic $(\mu_t)_{a\leq t\leq b}$,
    \begin{equation*}
        \lim\limits_{t\searrow a}\frac{1}{t-a}\bigg(\mathbb{H}\big[\mu_t|Q\big]-\mathbb{H}\big[\mu_a|Q\big]\bigg)=\frac{1}{b-a}\mathbb{E}^{\mu}\big[\nabla\mathcal{R}^\mu_a\cdot(\mathcal{T}-Id_{\mathbb{R}^d})(X_a)\big]
    \end{equation*}
\end{lemma}
\begin{proof}
    Only in this proof, we use $V^\mu_t=(V^{\mu,(1)}_t\ldots,V^{\mu,(d)}_t)$ to denote the velocity field defined by,
    \begin{equation*}
        (t,x)\mapsto V^\mu_t(x)\coloneqq(\mathcal{T}-Id_{\mathbb{R}^d})\big((\mathcal{T}_t)^{-1}x\big),\quad\text{for all}\quad(t,x)\in[a,b]\times\mathbb{R}^d.
    \end{equation*}
    Then, $V^\mu_t(\cdot)$ is associated to the transport $\mathcal{T}_t$ in the sense that
    \begin{equation*}
        \mathcal{T}_t(x)=x+\frac{1}{b-a}\int_a^t V^\mu_s\big(\mathcal{T}_s(x)\big)\,ds,\quad\text{for all}\quad a\leq t\leq b.
    \end{equation*}
    Since each $\mu_t$ is absolutely continuous with respect to $Q$, while $Q$ is absolutely continuous with respect to the Lebesgue measure on $\mathbb{R}^d$, then $\mu_t$ is absolutely continuous with respect to the Lebesgue measure on $\mathbb{R}^d$ with density $\rho^\mu_t(\cdot):\mathbb{R}^d\to\mathbb{R}_+$. According to \cite[Theorem 5.34]{Villani},
    \begin{equation*}
        -\frac{\partial\rho^\mu_t}{\partial t}(x)=\sum\limits_{1\leq i\leq d}\frac{\partial V^{\mu,(i)}_t}{\partial x_i}(x)\rho^\mu_t(x)+\big(V^\mu_t\cdot\nabla\rho^\mu_t\big)(x),\quad\text{for all}\quad(t,x)\in[a,b]\times\mathbb{R}^d.
    \end{equation*}
    Recall that $(X_t)_{a\leq t\leq b}$ denotes the coordinate process in $\mathcal{C}([a,b];\mathbb{R}^d)$ with $X_a\sim\mu_a$. Then the integral form,
    \begin{equation*}
        X_t=X_a+\frac{1}{b-a}\int_a^tV^{\mu}_s(X_s)\,ds,\quad\text{for all}\quad a\leq t\leq b,
    \end{equation*}
    characterizes the law of $X_t$ satisfying $X_t\sim(\mathcal{T}_t)_\#\mu_a$ for all $a\leq t\leq b$. Hence,
    \begin{equation*}
        d\rho^\mu_t(X_t)=\frac{\partial\rho^\mu_t}{\partial t}(X_t)\,dt+\nabla\rho^\mu_t(X_t)\,dX_t=-\frac{1}{b-a}\sum\limits_{1\leq i\leq d}\frac{\partial V^{\mu,(i)}_t}{\partial x_i}(X_t)\rho^\mu_t(X_t)\,dt.
    \end{equation*}
    Therefore, $d\log\rho^\mu_t(X_t)=-(b-a)^{-1}\sum_{1\leq i\leq d}(\partial V^{\mu,(i)}_t/\partial x_i)(X_t)\,dt$, $a\leq t\leq b$. Since $q(\cdot)=\exp(-2\psi(\cdot))$,
    \begin{equation*}
        d\log q(X_t)=-2\nabla\psi(X_t)\,dX_t=-\frac{2}{b-a}(\nabla\psi\cdot V^{\mu}_t)(X_t)\,dt.
    \end{equation*}
    Henceforth,
    \begin{equation*}
        d\mathcal{R}^\mu_t(X_t)=d\log\frac{\rho^\mu_t}{q}(X_t)=\frac{1}{b-a}\big(2(\nabla\psi\cdot V^{\mu}_t)-\sum\limits_{1\leq i\leq d}\frac{\partial V^{\mu,(i)}_t}{\partial x_i}\big)(X_t)\,dt,\quad\text{for all}\quad a\leq t\leq b.
    \end{equation*}
    Taking $\mu$-expectation,
    \begin{equation*}
        \mathbb{H}\big[\mu_t|Q\big]-\mathbb{H}\big[\mu_a|Q\big]=\mathbb{E}^{\mu}\big[\mathcal{R}^\mu_t(X_t)\big]-\mathbb{E}^{\mu}\big[\mathcal{R}^\mu_a(X_a)\big]=\frac{1}{b-a}\mathbb{E}^{\mu}\big[\int_a^t\big(2(\nabla\psi\cdot V^{\mu}_t)-\sum\limits_{1\leq i\leq d}\frac{\partial V^{\mu,(i)}_t}{\partial x_i}\big)(X_s)\,ds\big],
    \end{equation*}
    for all $a\leq t\leq b$. Consequently,
    \begin{equation*}
        \lim\limits_{t\searrow a}\frac{b-a}{t-a}\bigg(\mathbb{H}\big[\mu_t|Q\big]-\mathbb{H}\big[\mu_a|Q\big]\bigg)=\mathbb{E}^{\mu}\big[2(\nabla\psi\cdot V^{\mu}_t)(X_a)-\sum\limits_{1\leq i\leq d}\frac{\partial V^{\mu,(i)}_t}{\partial x_i}(X_a)\big]=\mathbb{E}^{\mu_a}\big[(\nabla\mathcal{R}^{\mu}_a\cdot V^\mu_a)(X_a)\big],
    \end{equation*}
    where the last equality is due to integration by parts. Since $V^\mu_a=\mathcal{T}-Id_{\mathbb{R}^d}$, the assertion is verified.
\end{proof}
Now we impose the displacement convexity to the potential $\psi(\cdot):\mathbb{R}^d\to\mathbb{R}_+$, which is necessary to the formulation of the following lemma.
\begin{lemma}\label{lem: 6.7}
    Suppose the potential $\psi(\cdot)$ satisfies the curvature bound, $\text{Hess}(\psi)\geq\kappa Id_{\mathbb{R}^d}$, for some $\kappa>0$. Fix $\mu_a,\mu_b\in\mathscr{P}_2(\mathbb{R}^d)$ such that both are absolutely continuous with respect to $Q$. Then we have,
    \begin{equation*}
        \mathbb{H}\big[\mu_a|Q\big]-\mathbb{H}\big[\mu_b|Q\big]\leq-\mathbb{E}^{\mu}\big[\nabla\mathcal{R}^\mu_a\cdot(\mathcal{T}-Id_{\mathbb{R}^d})(X_a)\big]-\frac{\kappa}{2}W_2(\mu_a,\mu_b)^2.
    \end{equation*}
\end{lemma}
\begin{proof}
    Only in this proof, we define the following two functions,
    \begin{equation*}
        \mathscr{F}(t)\coloneqq\int_{\mathbb{R}^d}\rho^\mu_t(x)\log\rho^\mu_t(x)\,dx\qquad\text{and}\qquad\mathscr{H}(t)\coloneqq\int_{\mathbb{R}^d}2\psi(x)\rho^\mu_t(x)\,dx,\quad\text{for all}\quad a\leq t\leq b.
    \end{equation*}
    By \cite[Theorem 5.15]{Villani}, the functions $\mathscr{F}(t)$ and $\mathscr{H}(t)$ are, respectively, displacement convex and $\kappa$-uniformly displacement convex from the Wasserstein space perspective, i.e.
    \begin{equation*}
        \frac{\partial^2}{\partial t^2}\mathscr{F}(t)\geq0\qquad\text{and}\qquad\frac{\partial^2}{\partial t^2}\mathscr{H}(t)\geq\frac{\kappa}{(b-a)^2} W_2(\mu_a,\mu_b)^2,\quad\text{for all}\quad a\leq t\leq b.
    \end{equation*}
    Notice that $\mathbb{H}[\mu_t|Q]=\mathscr{F}(t)+\mathscr{H}(t)$, $a\leq t\leq b$. Henceforth, the relative entropy function $t\mapsto\mathbb{H}[\mu_t|Q]$ is $\kappa$-uniformly displacement convex from the Wasserstein space perspective, i.e.
    \begin{equation*}
        \frac{\partial^2}{\partial t^2}\mathbb{H}\big[\mu_t|Q\big]\geq\frac{\kappa}{(b-a)^2} W_2(\mu_a,\mu_b)^2,\quad\text{for all}\quad a\leq t\leq b.
    \end{equation*}
    Use Lemma \ref{lem: 6.6} and the Taylor formula, we observe that
    \begin{equation*}
        \mathbb{H}\big[\mu_b|Q\big]=\mathbb{H}\big[\mu_a|Q\big]+(b-a)\frac{\partial}{\partial t}\mathbb{H}\big[\mu_t|Q\big]\bigg|_{t=0^+}+\int_a^b(b-t)\frac{\partial^2}{\partial t^2}\mathbb{H}\big[\mu_t|Q\big]\,dt,
    \end{equation*}
    which implies
    \begin{equation*}
        \mathbb{H}\big[\mu_b|Q\big]-\mathbb{H}\big[\mu_a|Q\big]\geq\mathbb{E}^{\mu}\big[\nabla\mathcal{R}^\mu_a\cdot(\mathcal{T}-Id_{\mathbb{R}^d})(X_a)\big]+\frac{\kappa}{2}W_2(\mu_a,\mu_b)^2.
    \end{equation*}
    And the assertion is verified.
\end{proof}
Lemma \ref{lem: 6.7} is of the HWI inequality type of Cordero-Erausquin \cite{Cordero-Erausquin} and Otto/Villani \cite{Otto/Villani}, which relates the fundamental quantities of relative entropy (H), quadratic Wasserstein distance (W), and Fisher information (I). This inequality has also been discussed by Datta/Rouzé \cite{Datta/Rouzé}, McCann \cite{McCann}, and Villani \cite{Villani2} on its convexity results and on its relations to the relative entropy dissipation. Another exposition on the HWI inequalities is Gentil/Léonard/Ripani/Tamanini \cite{Gentil/Léonard/Ripani/Tamanini}. 
\begin{theorem}
    \label{exponential entropy decay, perturbed}
    Suppose the potential $\psi(\cdot)$ satisfies the curvature bound, $\text{Hess}(\psi)\geq\kappa Id_{\mathbb{R}^d}$ for some $\kappa>0$. Then the relative entropy decays exponentially. In particular,
    \begin{equation}
    \label{exponential decay, perturbed}
        \mathbb{H}\big[P^0_t|Q\big]\leq\mathbb{H}\big[P^0_{t_0}|Q\big]e^{-\kappa(t-t_0)},\quad\text{for all}\quad t\geq t_0.
    \end{equation}
\end{theorem}
\begin{proof}
    On the strength of the Cauchy-Schwarz inequality,
    \begin{equation*}
        -\mathbb{E}^{\mu}\big[\nabla\mathcal{R}^\mu_a\cdot(\mathcal{T}-Id_{\mathbb{R}^d})(X_a)\big]\leq\normy{\nabla\mathcal{R}^\mu_a(X_a)}_{L^2(\mu_a)}\normy{(\mathcal{T}-Id_{\mathbb{R}^d})(X_a)}_{L^2(\mu_a)}.
    \end{equation*}
    Use Lemma \ref{lem: 6.7}, the definition of Fisher information (\ref{Fisher information}) as well as the optimal transport property (\ref{optimal transport}),
    \begin{equation}\label{eqn, middle step}
        \mathbb{H}\big[\mu_a|Q\big]-\mathbb{H}\big[\mu_b|Q\big]\leq W_2(\mu_a,\mu_b)\sqrt{\mathbb{I}\big[\mu_a|Q\big]}-\frac{\kappa}{2}W_2(\mu_a,\mu_b)^2.
    \end{equation}
    Reading (\ref{eqn, middle step}), we first take $(\mu_a,\mu_b)=(Q,P^0_t)$ and then take $(\mu_a,\mu_b)=(P^0_t,Q)$, then
    \begin{equation*}
        \mathbb{H}\big[P^0_t|Q\big]\leq\frac{1}{2\kappa}\mathbb{I}\big[P^0_t|Q\big],\quad\text{for all}\quad t\geq0.
    \end{equation*}
    Applying Corollary \ref{cor: classical result, unperturbed},
    \begin{equation*}
        \frac{\partial}{\partial t}\mathbb{H}\big[P^0_t|Q\big]\leq-\kappa\mathbb{H}\big[P^0_t|Q\big].
    \end{equation*}
    And the assertion is verified.
\end{proof}
In light of the Bakry-Émery theory, therefore, Theorem \ref{exponential entropy decay, perturbed} yields the exponential decay rate of the relative entropy corresponding to the unperturbed Itô-Langevin dynamics.

}

\bibliographystyle{plain}
\bibliography{literature}
\begin{spacing}{1}

\end{spacing}

\end{document}